\documentclass[11pt]{article}
\usepackage{mathrsfs}
\usepackage{amsfonts}
\usepackage{amssymb,bbm}  
\usepackage{amsmath,amscd,graphicx,amsthm}
\usepackage{subeqnarray}
\usepackage{cases}
\usepackage{color}
\usepackage{paralist}
\usepackage{graphicx}
\usepackage{float}
\usepackage{epstopdf}  
\usepackage{amsmath}
\usepackage[numbers,sort&compress]{natbib}
\usepackage{indentfirst}
\usepackage{amsthm}
\usepackage{diagbox}
\usepackage{graphicx}
\usepackage{caption}
\usepackage{float}

\makeatletter \@addtoreset{equation}{section}

       \newtheorem{lem}{\bf Lemma}[section]
       \newtheorem{remark}{\bf Remark}[section]
       \newtheorem{expl}{\bf Example}[section]
       \newtheorem{assp}{\bf Assumption}
       \newtheorem{coro}{\bf Corollary}[section]
       
       \newtheorem{thm}{\bf Theorem}[section]

\newcommand{\RR}{\mathbb{R}}

\def\L{{\cal L}}

\def\be{\begin{equation}}     \def\ee{\end{equation}}
\def\bea{\begin{eqnarray}}    \def\eea{\end{eqnarray}}
\def\beaa{\begin{eqnarray*}}  \def\eeaa{\end{eqnarray*}}
\def\nn{\nonumber}
 
\begin{document}
\title{ The convergence of EM scheme in empirical approximation of  invariant probability measure for McKean-Vlasov SDEs.
}
\author{Yuanping Cui\thanks{School of Mathematics Sciences,
Tiangong  University, Tianjin, 300387, China. }
 \and Xiaoyue Li\thanks{ School of Mathematics Sciences, Tiangong  University, Tianjin, 300387, China. Research
of this author  was supported by the National Natural Science Foundation of China (No. 12371402, 11971096), the National Key R$\&$D Program of China (2020YFA0714102), the Natural Science Foundation of Jilin Province (No. YDZJ202101ZYTS154), and the Fundamental Research Funds for the Central Universities (2412021ZD013).  
}
}

\date{}

\maketitle
\begin{abstract}
Based on the assumption of the existence and uniqueness of the invariant measure for McKean-Vlasov stochastic differential equations~(MV-SDEs), a self-interacting process that depends only on the current and historical information of the solution is constructed for MV-SDEs. The convergence rate of the weighted empirical measure of the self-interacting process and the invariant measure of MV-SDEs is obtained in the $\mathcal{W}_2$-Wasserstein metric. Furthermore, under the condition of linear growth, an EM scheme whose uniformly 1/2-order convergence rate with respect to time is obtained is constructed for the self-interacting process. Then, the convergence rate between the weighted empirical measure of the EM numerical solution of the self-interacting process and the invariant measure of MV-SDEs is derived. Moreover, the convergence rate between the averaged weighted empirical measure of the EM numerical solution of the corresponding multi-particle system and the invariant measure of MV-SDEs in the $\mathcal{W}_2$-Wasserstein metric is also given. In addition, the computational cost of the two approximation methods is compared, which shows that the averaged weighted empirical approximation of the particle system has a lower cost. Finally, the theoretical results are validated through numerical experiments.


\end{abstract}

\section{Introduction}\label{sec:intr}
This paper  considers a class of Mckean-Vlasov SDEs (MV-SDEs) with the form of
\begin{equation}\label{eq1}
\mathrm{d}X_t=f(X_t,\mathcal{L}^{X}_{t})dt+g(X_t,\mathcal{L}^{X}_{t})\mathrm{d}B_{t},~~~\forall t\geq0,
\end{equation}
 where
 $$f:\RR^{d}\times \mathscr{P}_{2}(\RR^{d})\rightarrow\RR^{d},~~~g:\RR^{d}\times \mathscr{P}_{2}(\RR^{d})\rightarrow \RR^{d\times m},$$
and $\mathcal{L}^{X}_{t}$ denotes the marginal distribution  of the solution at time $t$.  In contrast to classical SDEs, the coefficients of MV-SDEs may depend not only on the current state  but also on the current distribution law of the solution,  leading to  the marginal distribution $\mathcal{L}^{X}_{t}$  satisfies a non-linear Fokker-Planck equation \cite{Wang2018, MR1512678,  MR0221595}, which also brings a nontrivial additional difficulty in studying the  invariant measure of MV-SDEs compared to classical SDEs \cite{MR4404942, Ahmed}. Earlier studies on the existence and uniqueness of MV-SDEs require a uniform dissipative condition, as seen in \cite{Ahmed, Wang2018}, or the dissipative conditions in long distance, as in \cite{Veretennikov, Liang, MR4312362}. Under Lyapunov and some monotone conditions, Wang \cite{MR4550214} recently proved exponential ergodicity for a class of non-dissipative McKean-Vlasov SDEs  in the  Wasserstein quasi-distance. In contrast,  some sufficient conditions without certain
monotonicity  are proposed  for the existence only in \cite{MR4260494, MR4404942,Ahmed}. In general, solving the nonlinear Fokker-Planck equation directly to obtain the explicit form of the invariant probability measure of MV-SDEs is not feasible. On the other hand,  the so called propagation of chaos states that McKean-Vlasov SDEs can be viewed as the limit of an individual
particle in the  interacting particle system
\begin{align}\label{eq1.2}
dX^{j,N}_{t}=f(X^{j,N}_{t},\mathcal{L}^{X,N}_{t})\mathrm{d}t+g(X^{j,N}_{t},\mathcal{L}^{X,N}_{t})\mathrm{d}B^{i}_{t},~~\mathcal{L}^{X,N}_{t}=\frac{1}{N}\sum_{j=1}^{N}\delta_{X^{j,N}_{t}},~~1\leq j\leq N,
\end{align}
as particle number $N\rightarrow\infty$, where $\{B^{j}_{t},t\geq0\}_{j=1}^{N}$ are independent identically distributed (i.i.d.) copies of Brownian motion $\{B_{t}\}_{t\geq0}$. 
Thus, if one derives the propagation of chaos theory  in infinite time interval, then we can approximate the invariant probability measure of MV-SDE by the empirical measure $\mathcal{L}^{X,N}_{t}$ as $N\rightarrow\infty$ and $t\rightarrow\infty$ in theory. This also doesn't provide an  explicit form of the invariant probability measure.
Thus, it is imperative to develop the numerical  approximation theory of the invariant probability  measure of MV-SDEs. 

For numerical approximation of MV-SDEs, an extra difficulty  offered over standard SDEs is the requirement to approximate the distribution law $\mathcal{L}^{X}_{t}$ at each time step. 
Based on the propagation of chaos,  approximating
the associated interacting particle system, known as the stochastic particle method, is currently
most common approach to numerically solve MV-SDEs. By the stochastic particle method, some progress  has been made in the study of the numerical method of MV-SDEs. Under Lipschitz regularity, the EM scheme with 1/2-order convergence rate in time step size is applied to approximate Mckean-Vlasov equation, see \cite{Bossy,MR1910635}. Under the linear growth conditions,  the Euler-Maruyama (EM) scheme  has also been used for the numerical simulation of MV-SDEs with irregular coefficients \cite{ MR4414415,MR4289889}. However, it also has been extensively discussed that the EM scheme runs into difficulties  in the  super-linear growth coefficients for standard SDEs \cite{MR2795791} or MV-SDEs \cite{Dos Reis}. Despite implicit methods being proven to be applicable for approximating super-linear MV-SDEs \cite{MR4413221,Dos Reis},  requiring solving a fixed-point equation at each time step leads to implicit schemes having a slower speed and a higher computational cost, especially in higher dimensions. Then, a series of improved EM methods applicable to approximating nonlinear SDEs are continuously being developed.  The tamed EM scheme with a $1/2$-order convergence  rate in time step size  was initially proposed by Dos Reis \cite{Dos Reis}  for MV-SDEs with super-linear growth drift with respect to the state variable. Recently, it has been extended to MV-SDEs with common noise and super-linear drift and diffusion coefficients to the states. Furthermore, the taming idea was further applied to the higher-order schemes; see  tamed Milstein  scheme \cite{MR4497846, MR4302574,MR4212406}.  Compared with the works for the finite time strong convergence analysis of the numerical method, the numerical approximation of the invariant probability measure of Mckean-Vlasov SDEs is rather scarce. Using the stochastic particle method and EM scheme, the ergodic measure approximation of MV-SDEs with additive noise and linearly bounded drift term can be referred to \cite{Vere}.

Although the stochastic particle method is a feasible way to approximate the distribution law in coefficients of MV-SDEs,  there are still some limitations. The particle system is fully coupled and dependent on particle number $N$, implying that a sufficiently large $N$ must be preset  in numerical computation to achieve the expected accuracy. Meanwhile, changing $N$ requires recomputing  the whole particle system, which increases the computational burden; see \cite{K.D, MR4580925}. Furthermore,  the more particles in this system, the greater the risk of system divergence \cite{Dos Reis}.   It should be pointed out that  the approximation of the invariant probability  measure of MV-SDE requires simulating  particle systems over a long-time horizon, which may further exacerbate the computational burden. Thus, it is desirable to establish a new numerical approximation theory of the invariant probability measures for MV-SDEs.

Recently, based on the existence and uniqueness of invariant probability measure of MV-SDEs, Du-Jiang-Li \cite{MR4580925} used the weighted empirical measures of some self-interacting processes, whose coefficients depend only on the current and  historical information, to approximate the invariant probability  measure of the ergodic MV-SDEs in theory. This theory allows us to approximate the invariant probability measure of MV-SDEs by designing numerical schemes to simulate only one path of the corresponding self-interacting process.
Since this approximation  utilizes the current and history information, the approximation accuracy can be improved by adding the time $t$ until reaching the desired level. That is, there is no need to set the time $t$ in advance. With this new perspective,  we aim to  construct   numerical schemes for the self-interacting processes to approximate the  invariant probability measure of MV-SDEs.

This paper is based on the recent work  [1] and consists of two parts.  First,  we establish the convergence rate of the weighted empirical measure of the self-interacting process and the invariant probability measure of MV-SDEs in the $\mathcal{W}_2$-Wasserstein metric. Then, we design the appropriate EM schemes for the self-interacting process and  get the 1/2-order uniform convergence rates with respect to time. Based on these results,  we further derive the convergence rate between the weighted empirical measure of the EM numerical solution of the self-interacting process and the invariant probability measure of MV-SDEs. On the other hand,   we also provide the convergence rate between the averaged weighted empirical measure of the EM numerical solution of the corresponding multi-particle system and the invariant probability measure of MV-SDEs in the $\mathcal{W}_2$-Wasserstein metric. Compared with the generic results obtained in \cite{MR4580925}, our convergence rate of empirical approximation  is optimized.  More importantly, it should be emphasized that our research focuses on  designing implementable algorithms to approximate the invariant probability measures of MV-SDEs, which differs from the work in \cite{MR4580925}.

The material of this paper is organized as follows. Section 2 introduces some notations and preliminaries. Section 3 gives  the convergence rate of the weighed and the averaged weighed empirical approximations  to  the invariant measures of MV-SDEs, respectively. Section 4 gives error estimations of the weighted and averaged weighted empirical measures of the relevant  numerical solutions and invariant measures of MV-SDEs. Furthermore, section 4 discusses the computational cost of the EM method in different empirical approximations. Section 5 presents numerical examples to verify our theory. 

\section{Preliminary}\label{s-c2}
 Let $( \Omega,\cal{F},\mathbb{P})$ be a complete  probability space with a natural filtration $\{\mathcal{F}_{t}\}_{t\geq 0}$ satisfying the usual conditions. Let $|\cdot|$ denote  the Euclidean norm in $\RR^d$ and the Frobenius norm in $\RR^{d\times m}.$
If $A$ is a vector or matrix, the transpose is denoted by $A^T$. For any continuous function $\psi$ on $\RR^{d}$, $\|\psi\|_{\infty}=\sup_{x\in \RR^{d}}|\psi(x)|$. 
 For any $x\in \RR^{d}$, let $\boldsymbol{\delta}_{x}$ stand for the Dirac measure centred at  $x\in \RR^{d}$.
For a set $\mathbb{D}$, denote its indicator function  by $I_{\mathbb{D}}$, namely,  $I_{\mathbb{D}}(x)=1$ if $x\in \mathbb{D}$ and $0$ otherwise. For any $a,b\in \mathbb{R}$,  $a\vee b: =\max\{a,b\}$ and $a\wedge b:=\min\{a,b\}$. Let  $\mathbb{N}=\{0, 1, 2,\cdots\}$ be the set of all non-negative integers and $\mathbb{N}_{+}=\{1,2,3,\cdot\cdot\cdot\}$ be the set of positive intergers. For any  $a>0$ and $\tau>0$, let $\lfloor a \rfloor$ be  the integer part of $a$ and  define $\lfloor a \rfloor_{\tau}=\lfloor a/\tau\rfloor$. For a stochastic process $H:=\{H_{t}\}_{t\geq0}$, use $\mathcal{L}^{H}_{t}$ to denote the distribution law of $H$ at time $t$.
For convenience, let $C$  and $C_{p}$  denote two generic positive real constants, respectively, whose values may change in different appearances, where $C_{p}$ is used to emphasize the dependence of  the constant $C$  on $p$. In addition, the generic constants $C$ and $C_{p}$ are independent of parameters $\Delta$, $\tau$ and $M$ that occur in the later sections.
 
 Given the measurable space $(\RR^{d}, \mathcal{B}(\RR^{d}))$,  denote by $\mathscr{P}(\RR^{d})$ the set of all the  probability measures on
this space. For any $q\geq2$, let $\mathscr{P}_{q}(\RR^{d})$ be the set of  probability measure $\mu\in \mathscr{P}(\mathbb{R}^{d})$ which satisfies 
\begin{align}\label{eqc2.1}
\mu(|\cdot|^{q}):=\int_{\RR^{d}}|x|^q\mu(\mathrm{d}x)<\infty.
\end{align}
Define the $\mathcal{W}_2$ -Wasserstein distance on $\mathscr{P}_{2}(\RR^{d})$ by
\begin{align}\label{eqc2.3}
\mathcal{W}_{2}(\mu_1,\mu_2):=\inf_{\pi\in \mathcal{C}(\mu_1,\mu_2)}\Big(\int_{\mathbb{R}^{d}\times \mathbb{R}^{d}}|x-y|^2\pi(\mathrm{d}x,\mathrm{d}y)\Big)^{\frac{1}{2}},
\end{align}
where $\mathcal{C}(\mu_1,\mu_2)$ stands for the set of all probability measures on $\mathbb{R}^{d}\times \mathbb{R}^{d}$ with respective marginals $\mu_1$ and $\mu_2$. It is well known from \cite[Lemma 5.3, Theorem 5.4]{MR2091955} that $\mathscr{P}_{2}(\RR^{d})$ is a Polish space under the $\mathcal{W}_2$-Wasserstein distance. For any $\mu\in \mathscr{P}(\RR^{d})$ and the function $f$ on $\RR^{d}$, define $$\mu(f):=\int_{\RR^{d}}f(x)\mu(\mathrm{d}x).$$
Furthermore, for any $\mu\in \mathscr{P}_{2}(\RR^{d})$,  define $Var(\mu):=\mu(|x|^{2})-|\mu(x)|^{2}$. For convenience,  for any $\RR^{d}$-valued stochastic process $H:=\{H_{t}\}_{t\geq0}$ and $\nu\in \mathscr{P}(\RR^{d})$, we use $H^{\nu}:=\{H^{\nu}_{t}\}_{t\geq0}$ to stress the initial distribution of process $H$ is $\nu$.

Now we state the assumptions throughout this paper.
\begin{assp}\label{as2}
The coefficients $f$ and $g$ are continuous on $\RR^{d}\times \mathscr{P}_{2}(\RR^{d})$ and bounded on any bounded set in $\RR^{d}\times \mathscr{P}_{2}(\RR^{d})$. Furthermore,
there are  constants $\kappa_1>\kappa_2>0$, $\bar{\kappa}_1>\bar{\kappa}_2>0$ and  $\rho>0$ such that for any $x, x_1,x_2\in \RR^{d}$ and $\mu, \mu_1, \mu_2\in \mathscr{P}_{2}(\RR^{d})$,
\begin{align}\label{e2.1}
&2x^{T}f(x,\mu)+(1+\rho)|g(x,\mu)|^2\leq -\kappa_1|x|^2+\kappa_2\mu(|\cdot|^2)+C,
\end{align}
and
\begin{align}\label{e2.2}
&2(x_1-x_2)^{T}\big(f(x_1,\mu_1)-f(x_2,\mu_2)\big)+|g(x_1,\mu_1)-g(x_2,\mu_2)|^2\nn\
\\~~~\leq& -\bar{\kappa}_1|x_1-x_2|^2+\bar{\kappa}_2\mathcal{W}^2_{2}(\mu_1,\mu_2).
\end{align}
\end{assp}
Under Assumption \ref{as2}, the well-posedness of a strong solution to MV-SDE \eqref{eq1} can be found in \cite{MR4580925, Wang2018, MR4497846}. Furthermore,
it can be  proved as in \cite{Wang2018} that under Assumption 1, the solution  $X$ to MV-SDE \eqref{eq1}  has some essential properties  illustrated in the subsequent lemma.
\begin{lem}\label{L1}
Under Assumption \ref{as2},  for any initial distribution $\nu\in \mathscr{P}_{2+\rho}(\RR^{d})$, MV-SDE \eqref{eq1} has a unique strong solution $X^{\nu}_{t}$ on $t\geq0$, satisfying
\begin{align}\label{eqq2.1}
\sup_{0\leq t< \infty}\mathbb{E}|X^{\nu}_{t}|^{2+\rho}< \infty.
\end{align}
Furthermore, the solution $X^{\nu}_{t}$ has a unique invariant  probability measure $\mu^{*}\in \mathscr{P}_{2+\rho}(\RR^{d})$ such that
\begin{align}\label{eqq2.2}
\mathcal{W}^{2}_{2}(\mathcal{L}^{X^{\nu}}_{t},\mu^*)\leq \mathcal{W}^{2}_{2}(\nu,\mu^*)e^{-(\bar{\kappa}_1-\bar{\kappa}_2)t},~~~\forall t\geq0.
\end{align}
\end{lem}
Based on the existence and uniqueness of the invariant probability measure of MV-SDE \eqref{eq1},  we take $f(x, \mu)=$ $f\left(x, \mu^*\right)$ and $g(x, \mu)=g\left(x, \mu^*\right)$ to  obtain 
\begin{align}\label{e2.2}
\mathrm{d}\tilde{X}_{t}=f(\tilde{X}_{t},\mu^{*})\mathrm{d}t+g(\tilde{X}_{t},\mu^{*})\mathrm{d}B_{t},
\end{align}
which is exactly a classical SDE  with the Markovian property. Let $\mathbb{P}_t$ be the associated semigroup of $\tilde{X}$, and $\tilde{X}^{x}$ and $\tilde{X}^{\nu}$ be the solutions of the above equations with initial distribution laws $\delta_x$ and $\nu\in \mathscr{P}_{2+\rho}(\RR^{d})$, respectively. Obviously, according to the uniqueness of the strong solution of SDE \eqref{e2.2}, it is easy to derive that $\mu^*$ is also the unique  invariant probability measure of $\tilde{X}$. Next, we present a few well-known properties of SDE \eqref{e2.2} without the detailed proof.
\begin{lem}\label{L2}
Let Assumption \ref{as2} hold. Then
 \begin{itemize}
 \item[(1)] For any initial distribution $\nu\in \mathscr{P}_{2+\rho}(\RR^{d})$,
 \begin{align}\label{eq2.2}
 \sup_{t\geq0}\mathbb{E}|\tilde{X}^{\nu}_t|^{2+\rho}< C\left(1+\nu(|\cdot|^{2+\rho})\right),
 \end{align}
 where the constant $C$ is independent of $t$.
 \item[(2)] For any initial values $x,y\in \RR^{d}$,
 \begin{align}\label{eq2.3}
 \mathbb{E}|\tilde{X}^{x}_{t}-\tilde{X}^{y}_{t}|^{2}\leq C|x-y|^{2}e^{-\bar{\kappa}_1 t}.
 \end{align}
 \item[(3)] For any initial distribution $\nu\in \mathscr{P}_{2+\rho}(\RR^{d})$, 
 \begin{align}\label{eq2.4}
\mathcal{W}^2_{2}\left( \mathcal{L}^{\tilde{X}^{\nu}}_{t}, \mu^{*}\right)\leq \mathcal{W}^2_{2}(\nu,\mu^{*})e^{-\bar{\kappa}_1t}.
 \end{align}
 \end{itemize}
\end{lem}
To proceed, we present  some common notations and useful conclusions.  
 $$\Gamma:=\left\{\gamma:=(\gamma_{t})_{t\geq0}: \gamma_{t}\in \mathscr{P}([0,1])\right\},$$
 where $\mathscr{P}([0,1])$ is the set of all probability measures on interval $[0,1]$. 
For any $\tau>0$, define  $\pi_{\tau}\in \Gamma$ by 
 $$\pi_{\tau,t}(\cdot):=\frac{1}{\lfloor t\rfloor_{\tau}+1}\sum_{i=0}^{\lfloor t\rfloor_{\tau}}\boldsymbol{\delta}_{\frac{i\tau}{t}}(\cdot)\in \mathscr{P}([0,1]),~~~\forall t\geq0.$$
By a simple computation, it is easy to verify that for any $\delta\in (0,1-\bar{\kappa}_2/\bar{\kappa}_1)$, 
\begin{align}\label{cyp2.11}
\limsup_{t\rightarrow\infty}\int_{0}^{1}s^{-\delta}\pi_{\tau,t}(\mathrm{d}s)<\frac{\bar{\kappa}_1}{\bar{\kappa}_2}.
\end{align}
Then  for any $\delta\in (0,1-\bar{\kappa}_2/\bar{\kappa}_1)$, define a set $\mathcal{A}_{\delta}$ by
\begin{align}\label{cyp2.12}
\mathcal{A}_{\delta}=\left\{\tilde{\kappa}_2: \bar{\kappa}_2<\tilde{\kappa}_2<\bar{\kappa}_1~\text{and}~\limsup_{t\rightarrow\infty}\int_{0}^{1}s^{-\delta}\pi_{\tau,t}(\mathrm{d}s)< \frac{\bar{\kappa}_1}{\tilde{\kappa}_2}\right\}.
\end{align}
Thanks to \eqref{cyp2.11},  it is easily verified that $\mathcal{A}_{\delta}\neq \emptyset$.
Given $\tau>0$ and a $\RR^{d}$-valued  stochastic process $H=\{H_{t}\}_{t\in I}$, where $I=[0,\infty)$  or $I=\{k\tau\}_{k\in \mathbb{N}}$, we further define two kinds of  the  empirical measures of $H$ with respect to $\pi_{\tau}$ by
\begin{align}\label{eq2.11}
\mathcal{E}^{H}_{\tau,t}(\cdot):=\frac{1}{\lfloor t\rfloor_{\tau}+1}\sum_{i=0}^{\lfloor t\rfloor_{\tau}}\boldsymbol{\delta}_{H_{i\tau}}(\cdot)=\int_{0}^{1}\boldsymbol{\delta}_{H_{ts}}(\cdot)\pi_{\tau,t}(\mathrm{d}s)\in \mathcal{P}(\RR^{d}),~~~\forall t\in I,
\end{align}
and 
\begin{align}\label{eq2.22}
\mathcal{L}^{H}_{\tau,t}(\cdot):=\frac{1}{\lfloor t\rfloor_{\tau}+1}\sum_{i=0}^{\lfloor t\rfloor_{\tau}}\mathcal{L}^{H}_{i\tau}(\cdot)=\int_{0}^{1}\mathcal{L}^{H}_{ts}(\cdot)\pi_{\tau,t}(\mathrm{d}s)\in \mathcal{P}(\RR^{d}),~~~\forall t\in I,
\end{align}
respectively. Recalling \eqref{eqc2.1} and \eqref{eqc2.3}, we compute that
\begin{align}\label{eq2.9}
\mathcal{E}^{H}_{\tau,t}(|\cdot|^2)=\frac{1}{\lfloor t\rfloor_{\tau}+1}\sum_{i=0}^{\lfloor t\rfloor_{\tau}}|H_{i\tau}|^{2}=\int_{0}^{1}|H_{ts}|^{2}\pi_{\tau,t}(\mathrm{d}s), ~~~\forall t\in I,
\end{align}
and  for any $\RR^{d}$-valued  stochastic processes $H:=\{H_{t}\}_{t\geq0}$ and $\tilde{H}:=\{\tilde{H}_{t}\}_{t\geq0}$, 
\begin{align}\label{2.10}
\mathbb{E}\left[\mathcal{W}^2_{2}\left(\mathcal{E}^{H}_{\tau,t}, \mathcal{E}^{\tilde{H}}_{\tau,t}\right)\right]\leq \frac{1}{\lfloor t\rfloor_{\tau}+1}\sum_{i=0}^{\lfloor t\rfloor_{\tau}}\mathbb{E}|H_{i\tau}-\tilde{H}_{i\tau}|^2=\int_{0}^{1}\mathbb{E}|H_{ts}-\tilde{H}_{ts}|^2\pi_{\tau,t}(\mathrm{d}s),~~\forall t\in I.
\end{align}
\begin{lem}[{{\cite[Lemma 3.4]{MR4580925}}}]\label{L3}
Let $\xi$ be a random variable with distribution law $\Phi \in \mathscr{P}_2(\RR^{d})$. Then for any  $\mu\in \mathscr{P}_{2}(\RR^{d})$,
$$
\mathcal{W}^2_2(\Phi \times \mu, \mu) \leq \mathbb{E}|\xi|^2,
$$
where $``\times"$ denotes the convolution.
\end{lem}

Borrowing the idea of  \cite[Lemma 4.1]{MR4580925}, we give the following result.
\begin{lem}\label{L5}
Let  $\delta\in (0,1-\bar{\kappa}_2/\bar{\kappa}_1)$,  $0<\varepsilon<1$,  $\tilde{\kappa}_2\in \mathcal{A}_{\delta}$ and $V$ be a non-negative constant. If 
 a differentiable function $F:[0, \infty) \rightarrow [0, \infty)$ satisfies $F(0)=0$ and
$$
F^{\prime}(t) \leq-\bar{\kappa}_1F(t)+\tilde{\kappa}_2 \int_0^1 F(ts)\pi_{\tau,t}(\mathrm{d}s)+V(1\wedge t^{-\varepsilon}),
$$
then there exists a constant $C_{\delta}$ such that
$$
F(t) \leq C_{\delta}t^{-\delta\wedge\varepsilon},
$$
where the constant $C_{\delta}$ is dependent on $\bar{\kappa}_1$ and $\tilde{\kappa}_2$ but independent of $F$.
\end{lem}
\begin{proof}
For any $\delta\in (0,1-\bar{\kappa}_2/\bar{\kappa}_1)$, define
$$
G(t):=t^{-\delta\wedge\varepsilon}.
$$
Thanks to  \eqref{cyp2.12}, we have
$$
\limsup_{t \rightarrow \infty} \int_0^1 s^{-\delta} \pi_{\tau,t}(\mathrm{d}s):=L_{\delta}<\frac{\bar{\kappa}_1}{\tilde{\kappa}_2}.
$$
Thus,  there is a constants $T_1>1$  such that for any $t\geq T_1$, $$\int_{0}^{1}s^{-\delta}\pi_{\tau,t}(\mathrm{d}s)\leq \frac{\bar{\kappa}_1}{\tilde{\kappa}_2}-\iota,$$
 where $\iota=(\bar{\kappa}_1/\tilde{\kappa}_2-L_{\delta})/2$.
By a direct computation, we can further find a constant $T> T_1$ such that for any $t\geq T$,
$$(\delta\wedge\varepsilon) t^{-(\delta\wedge\varepsilon+1)}-\frac{\tilde{\kappa}_2\iota}{2} t^{-\delta\wedge\varepsilon}\leq0,$$
and
\begin{align}\label{cyp2.15}
\left(1+\frac{T^{\delta\wedge\varepsilon}}{\bar{\kappa}_1-\tilde{\kappa}_2}\right)\frac{\tilde{\kappa}_2\iota}{2}>1.
\end{align}
Then for any $t\geq T$,
$$
G^{\prime}(t) \geq-\bar{\kappa}_1 G(t)+\tilde{\kappa}_2 \int_0^1 G(t s) \pi_{\tau,t}(\mathrm{d} s)+\frac{\tilde{\kappa}_2 \iota}{2} t^{-\varepsilon}.
$$
Assume $t^* \in \arg \max _{0 \leq t \leq T} F(t)$, then we have $F^{\prime}\left(t^*\right) \geq 0$. Therefore, for any $t>1$
$$
0 \leq-\bar{\kappa}_1 F\left(t^*\right)+\tilde{\kappa}_2\int_0^1 F\left(t^* s\right)\pi_{\tau,t}(\mathrm{d}s)+V(1\wedge t^{-\varepsilon})\leq-(\bar{\kappa}_1-\tilde{\kappa}_2) F\left(t^*\right)+V,
$$
 which, together with the fact $\tilde{\kappa}_2<\bar{\kappa}_1$, implies that
$$
\max _{0 \leq t \leq T} F(t) \leq \frac{V}{\bar{\kappa}_1-\tilde{\kappa}_2}.
$$
Furthermore, for any $a>0$, define
$$
H_{a}(t):=\left(1+\frac{T^{\delta\wedge\varepsilon}}{\bar{\kappa}_1-\tilde{\kappa}_2}\right)V G(t)-F(t)+a.
$$
We derive from \eqref{cyp2.15} for any $t \geq T$,
$$
H^{\prime}_{a}(t)>-\bar{\kappa}_1 H_{a}(t)+\tilde{\kappa}_2 \int_0^1 H_{a}(ts)\pi_{\tau,t}(\mathrm{d} s),
$$
and $H_{a}(t) >a$ for any $0<t \leq T$. Now we claim that $H_{a}(t)>0$ holds for any $t \geq T$. Otherwise, let $T^*=\inf\{t\geq T, H_{a}(t)=0\}<\infty$. Therefore,
$$
H^{\prime}_{a}\left(T^*\right)>\tilde{\kappa}_2 \int_0^1 H_{a}(T^*s) \pi_{\tau,T^{*}}(\mathrm{d}s) \geq 0 .
$$
This indicates that there exists a constant $T^{* *} \in\left(T, T^*\right)$ such that $H_a\left(T^{* *}\right)=0$, which contradicts the definition of $T^*$. Thus, $H_{a}(t)>0$ for any $t\geq0$. This leads to
$$
F(t) \leq\left(1+\frac{T^{\varepsilon}}{\bar{\kappa}_1-\tilde{\kappa}_2}\right) V t^{-\varepsilon}+a.
$$
The arbitrariness of $a$ implies that
\begin{align*}
F(t)\leq \left(1+\frac{T^{\varepsilon}}{\bar{\kappa}_1-\tilde{\kappa}_2}\right)V t^{-\varepsilon}.
\end{align*}
The proof is complete.
\end{proof}

\begin{lem}\label{CLem2.5}
Let $0<\beta<\alpha$ and $0<V$. If  a real-valued  differentiable function $F$ satisfies that
$$
F^{\prime}(t) \leq-\alpha F(t)+\beta \int_0^1 F(ts)\pi_{\tau,t}(\mathrm{d}s)+V,
$$
then $$\sup_{t\geq 0}F(t)\leq \frac{V+\beta F(0)}{\alpha-\beta}\vee F(0).$$ In particular, if $F(0)=0$, then 
$$\sup_{t\geq0}F(t)\leq \frac{V}{\alpha-\beta}.$$
\end{lem}
\begin{proof}
If   $F'(0)\geq 0$, for any $T>0$, assume $t^* \in \arg \max _{0 \leq t \leq T} F(t)$, then we have $F^{\prime}\left(t^*\right) \geq 0$. Thus,
$$
0 \leq-\alpha F\left(t^*\right)+\beta \int_0^1 F\left(t^* s\right) \pi_{\tau,t^*}(\mathrm{d}s)+V \leq-(\alpha-\beta) F\left(t^*\right)+V,
$$
which implies that
$$
\max _{0 \leq t \leq T} F(t)=F(t^*) \leq \frac{V}{\alpha-\beta}.
$$
Thanks to the arbitrariness of $T$,  the desired result holds directly for the case $F'(0)\geq0$. \\
On the other hand, if $F'(0)<0$, then we can find a constant
$$T^*=\inf\{t\geq0, F'(t)=0\}<\infty, $$
where we set $\inf\emptyset=\infty $. Otherwise,  $\sup_{t\geq0}F(t)\leq F(0)$, which implies the desired result for the case $F'(0)<0$. In view of the definition $T^*$, we have 
\begin{align}\label{eq2.15}
\sup_{0\leq t\leq T^*}F(t)\leq F(0).
\end{align}
 For any $T\geq T^*$, assume that $t^* \in \arg \max _{T^* \leq t \leq T} F(t)$. Then we also  derive that 
\begin{align}
0\leq F'(t^*)&\leq -\alpha F(t^*)+\beta\int_{0}^{1}F(t^*s)\pi_{\tau, t^*}(\mathrm{d}s)+V\nn\
\\&\leq -\alpha F(t^*)+\beta(F(0)+F(t^*))+V.
\end{align}
By a simple computation yields that
\begin{align*}
\sup_{T^*\leq t\leq T}=F(t^*)\leq \frac{\beta F(0)+V}{\alpha-\beta}.
\end{align*}
This, together with the arbitrariness of $T$ and \eqref{eq2.15}, implies the desired result for the case $F'(0)<0$. The proof is complete.
\end{proof}
\section{EM scheme in the weighed empirical approximation}\label{S3}
This section will study the convergence of the EM scheme in the weighted empirical approximation of the MV-SDE. First, establishing the proper self-interaction process and using the weighted empirical measure of the self-interaction process to approximate the invariant measure of MV-SDE. Based on this, further for the self-interacting process to develop explicit EM
scheme to prove the convergence between the weighted empirical measure of the EM numerical solution and the invariant measure of the distribution-dependent stochastic differential equation.
\subsection{The empirical approximation of the self-interacting process}
 For any $\tau>0$, let $Y$ satisfy 
\begin{align}\label{e2.13}
\mathrm{d} Y_t=b\left(Y_t, \mathcal{E}^{Y}_{\tau,t}\right) \mathrm{d} t+\sigma\left(Y_t, \mathcal{E}^{Y}_{\tau,t}\right) \mathrm{d}B_t,~~\forall t\geq0,
\end{align}
with the same initial value as that of MV-SDE \eqref{eq1}, where the weighted empirical measure $\mathcal{E}^{Y}_{\tau,t}$ is defined by \eqref{eq2.11}. Under Assumption \ref{as2},  the well-posedness  of equation \eqref{e2.13}  has been studied in \cite[Lemma 7.2]{MR4580925}. 
We aim to study the convergence rate between  the weighted empirical measure $\mathcal{E}^{Y}_{\tau,t}$ and  the invariant probability measure $\mu^*$ of MV-SDE \eqref{eq1} in $\mathcal{W}_2$-Wasserstein distance.  For this, we first prove the convergence between the weighted empirical measure $\mathcal{E}^{\tilde{X}}_{\tau,t}$ of Markovian SDE \eqref{e2.2} and the invariant  probability measure $\mu^*$, then the convergence between the weighed empirical measures $\mathcal{E}^{\tilde{X}}_{\tau,t}$ and $\mathcal{E}^{Y}_{\tau,t}$, and thus obtain the desired convergence between $\mathcal{E}^{Y}_{\tau,t}$ and $\mu^*$.
Now, we analyze the former.
\begin{lem}\label{L2.4}
Let Assumption \ref{as2} hold. Then for any initial distribution $\nu\in \mathscr{P}_{2+\rho}(\RR^{d})$, there exists a constant $C_{\nu}$ such that
\begin{align}
\mathbb{E}\Big[\mathcal{W}^{2}_{2}\big(\mathcal{E}^{\tilde{X}^{\nu}}_{\tau,t},\mu^{*}\big)\Big]\leq C_{\nu} \left(1+\tau\right) t^{-\eta},
\end{align}
where, and from now on, $\eta=\rho/[(d+2)(\rho+2)]$. 
\end{lem}
\begin{proof}
 Let $\xi$ be a random vector in $\mathbb{R}^d$  with $|\xi| \leq 1$ and has a smooth density function.  Denote the distribution function of $r \xi$ by  $\Phi_r$ and  its density function by $\phi_r$. Then,  by convolution with $\Phi_r$, we modify the weighted empirical measures as 
$$
\mathcal{E}^{\tilde{X}^{\nu}}_{\tau,t,r}:=\Phi_r \times \mathcal{E}^{\tilde{X}^{\nu}}_{\tau,t}=\frac{1}{\lfloor t\rfloor_{\tau}+1}\sum_{i=0}^{\lfloor t\rfloor_{\tau}}\Phi_r \times \boldsymbol{\delta}_{\tilde{X}^{\nu}_{i\tau}}
$$
and
$$
\mathcal{L}^{\tilde{X}^{\nu}}_{\tau,t,r}:= \Phi_r\times \mathcal{L}^{\tilde{X}^{\nu}}_{\tau,t}=\frac{1}{\lfloor t\rfloor_{\tau} +1}\sum_{i=0}^{\lfloor t\rfloor_{\tau}}\Phi_r\times \mathcal{L}^{\tilde{X}^{\nu}}_{i\tau}.
$$
By a simple computation, the density functions of $\mathcal{E}^{\tilde{X}^{\nu}}_{\tau,t,r}$ and $\mathcal{L}^{\tilde{X}^{\nu}}_{\tau,t,r}$ are given by
$$
\frac{1}{\lfloor t\rfloor_{\tau}+1}\sum_{i=0}^{\lfloor t\rfloor_{\tau}} \phi_r(\cdot-\tilde{X}^{\nu}_{i\tau})  \quad \text { and } \quad \frac{1}{\lfloor t\rfloor_{\tau}+1}\sum_{i=0}^{\lfloor t\rfloor_{\tau}} \mathbb{E}\big[\phi_r(\cdot-\tilde{X}^{\nu}_{i\tau})\big],
$$
respectively. Then using the elementary inequality yields that
\begin{align}\label{e2.4}
\mathcal{W}_2\big(\mathcal{E}^{\tilde{X}^{\nu}}_{\tau,t}, \mu^*\big) \leq & \mathcal{W}_2\big(\mathcal{E}^{\tilde{X}^{\nu}}_{\tau,t}, \mathcal{E}^{\tilde{X}^{\nu}}_{\tau,t,r}\big)+\mathcal{W}_2\big(\mathcal{E}^{\tilde{X}^{\nu}}_{\tau, t,r}, \mathcal{L}^{\tilde{X}^{\nu}}_{\tau, t,r}\big) \nn\\
& +\mathcal{W}_2\big(\mathcal{L}^{\tilde{X}^{\nu}}_{\tau,t, r}, \mathcal{L}^{\tilde{X}^{\nu}}_{\tau,t}\big)+\mathcal{W}_2\big(\mathcal{L}^{\tilde{X}^{\nu}}_{\tau,t}, \mu^*\big).
\end{align}
By Lemma \ref{L3}, we have
\begin{align}\label{e2.5}
\mathcal{W}^2_2\big(\mathcal{E}^{\tilde{X}^{\nu}}_{\tau,t}, \mathcal{E}^{\tilde{X}^{\nu}}_{\tau,t,r}\big)+\mathcal{W}^2_2\big(\mathcal{L}^{\tilde{X}^{\nu}}_{\tau,t, r}, \mathcal{L}^{\tilde{X}^{\nu}}_{\tau,t}\big) \leq 2\mathbb{E}|r \xi|^2 \leq C r^2,~~~\mathrm{a.s.}
\end{align}
On the other hand, applying Lemma \ref{L2} shows that
\begin{align*}
\mathcal{W}^2_2\big(\mathcal{L}^{\tilde{X}^{\nu}}_{\tau,t}, \mu^*\big)& \leq \frac{1}{\lfloor t\rfloor_{\tau}+1}\sum_{i=0}^{\lfloor t\rfloor_{\tau}} \mathcal{W}^2_2\big(\mathcal{L}^{\tilde{X}^{\nu}}_{i\tau}, \mu^*\big) \\
& \leq \mathcal{W}^2_{2}\big(\nu,\mu^{*}\big)\frac{1}{\lfloor t\rfloor_{\tau}+1}\sum_{i=0}^{\lfloor t\rfloor_{\tau}} e^{-\bar{\kappa}_1 i\tau}\nn\
\\&\leq \mathcal{W}^2_{2}(\nu,\mu^{*})\frac{1}{\lfloor t\rfloor_{\tau}+1}\frac{1}{1-e^{-\bar{\kappa}_1\tau}}.
\end{align*}
Thanks to the inequality $1/(1-e^{-x})\leq 1+1/x$  for any $x>0$, we obtain  that there exists a constant $C_{\nu}$ such that
\begin{align}\label{e2.6}
\mathcal{W}^2_2\big(\mathcal{L}^{\tilde{X}^{\nu}}_{\tau,t}, \mu^*\big)& \leq \frac{C_{\nu}}{\lfloor t\rfloor_{\tau}+1}\big(1+\frac{1}{\bar{\kappa}_1\tau}\big)\leq C_{\nu}\big(\tau+\frac{1}{\bar{\kappa}_1}\big)\frac{1}{t}.
\end{align}
Next we focus on estimating the  term $$\mathbb{E}\Big[\mathcal{W}^2_{2}\big(\mathcal{E}^{\tilde{X}^{\nu}}_{\tau,t,r},\mathcal{L}^{\tilde{X}^{\nu}}_{\tau,t,r}\big)\Big].$$
 Employing the density coupling lemma \cite[Lemma 3.3]{MR4580925} leads to that
\begin{align}\label{e2.7}
& \mathbb{E}\Big[\mathcal{W}^2_2\big(\mathcal{E}^{\tilde{X}^{\nu}}_{\tau,t,r}, \mathcal{L}^{\tilde{X}^{\nu}}_{\tau,t,r }\big)\Big] \nn\\
\leq & C \mathbb{E} \int_{\mathbb{R}^d}|y|^2\Big|\frac{1}{\lfloor t\rfloor_{\tau}+1}\sum_{i=0}^{\lfloor t\rfloor_{\tau}}\big(\phi_{r}(y-\tilde{X}^{\nu}_{i\tau})-\mathbb{E}\phi_{r}(y-\tilde{X}^{\nu}_{i\tau})\big)\Big| \mathrm{d}y \leq J_1+J_2,
\end{align}
where 
\begin{align*}
&J_1=C \int_{|y|>R}|y|^2 \mathbb{E}\Big|\frac{1}{\lfloor t\rfloor_{\tau}+1}\sum_{i=0}^{\lfloor t\rfloor_{\tau}}\big(\phi_{r}(y-\tilde{X}^{\nu}_{i\tau})-\mathbb{E}\phi_{r}(y-\tilde{X}^{\nu}_{i\tau})\big)\Big|  \mathrm{d} y,
\\& J_2=C \int_{|y| \leq R}|y|^2 \mathbb{E}\Big|\frac{1}{\lfloor t\rfloor_{\tau}+1}\sum_{i=0}^{\lfloor t\rfloor_{\tau}}\big(\phi_{r}(y-\tilde{X}^{\nu}_{i\tau})-\mathbb{E}\phi_{r}(y-\tilde{X}^{\nu}_{i\tau})\big)\Big| \mathrm{d} y.
\end{align*}
Owing to the non-negativity of density functions, we obtain that
\begin{align}\label{eqJ}
J_1 & \leq  \frac{C}{\lfloor t\rfloor_{\tau}+1} \sum_{i=0}^{\lfloor t\rfloor_{\tau}} \int_{|y|>R}|y|^2\mathbb{E}\left[\phi_r(y-\tilde{X}^{\nu}_{i\tau})+\mathbb{E}\phi_r(y-\tilde{X}^{\nu}_{i\tau})\right]\mathrm{d} y \nn\\
& = \frac{2C}{\lfloor t\rfloor_{\tau}+1}\sum_{i=0}^{\lfloor t\rfloor_{\tau}} \int_{|y|>R} \mathbb{E}\left[|y|^2 \phi_r(y-\tilde{X}^{\nu}_{i\tau})\right] \mathrm{d} y  \nn\\
& \leq \frac{C}{\lfloor t\rfloor_{\tau}+1}\sum_{i=0}^{\lfloor t\rfloor_{\tau}}\int_{|y|>R} \mathbb{E}\left[|y-\tilde{X}^{\nu}_{i\tau}|^2\phi_r(y-\tilde{X}^{\nu}_{i\tau})+|\tilde{X}^{\nu}_{i\tau}|^2 \phi_r(y-\tilde{X}^{\nu}_{i\tau})\right] \mathrm{d} y  \nn\\
& \leq \frac{C}{\lfloor t\rfloor_{\tau}+1}\sum_{i=0}^{\lfloor t\rfloor_{\tau}} \int_{\mathbb{R}^d}|y|^2 \phi_r(y) \mathrm{d} y +\frac{C}{\lfloor t\rfloor_{\tau}+1}\sum_{i=0}^{\lfloor t\rfloor_{\tau}} \int_{|y|>R} \mathbb{E}\left[|\tilde{X}^{\nu}_{i\tau}|^2 \phi_r(y-\tilde{X}^{\nu}_{i\tau})\right] \mathrm{d} y \nn\\
& \leq C r^2+\frac{C}{\lfloor t\rfloor_{\tau}+1}\sum_{i=0}^{\lfloor t\rfloor_{\tau}} \int_{|y|>R} \mathbb{E}\left[|\tilde{X}^{\nu}_{i\tau}|^2 \phi_r(y-\tilde{X}^{\nu}_{i\tau})\right] \mathrm{d} y .
\end{align}
Since the support of $\phi_r$ is contained in the ball $B_r(0)$,  we have
$$
\begin{aligned}
\int_{|y|>R} \mathbb{E}\Big[|\tilde{X}^{\nu}_{i\tau}|^2 \phi_r(y-\tilde{X}^{\nu}_{i\tau})\Big] \mathrm{d} y
\leq & \mathbb{E}\Big(|\tilde{X}^{\nu}_{i\tau}|^2 \mathbf{I}_{\left\{|\tilde{X}^{\nu}_{i\tau}|>R-r\right\}} \int_{|y|>R} \phi_r(y-\tilde{X}^{\nu}_{i\tau}) \mathrm{d} y\Big)\\
\leq &\mathbb{E}\big(|\tilde{X}^{\nu}_{i\tau}|^2 \mathbf{I}_{\{|\tilde{X}^{\nu}_{i\tau}|>R-r\}}\big).
\end{aligned}
$$
By \eqref{eq2.2}, using the  H\"older inequality and the Markov inequality, we deduce that
$$
\begin{aligned}
\mathbb{E}\Big(|\tilde{X}^{\nu}_{i\tau}|^2 \mathbf{I}_{\{|\tilde{X}^{\nu}_{i\tau}|>R-r\}}\Big) & \leq\big(\mathbb{E}|\tilde{X}^{\nu}_{i\tau}|^{2+\rho}\big)^{\frac{2}{2+\rho}}\Big[\mathbb{E}\big( \mathbf{I}_{\{|\tilde{X}^{\nu}_{i\tau}|>R-r\}}\big)\Big]^{\frac{\rho}{2+\rho}} \leq C_{\nu}(R-r)^{-\rho} .
\end{aligned}
$$
Combining the above estimates implies that
\begin{align}\label{e2.8}
J_1 \leq C r^2+C_{\nu}(R-r)^{-\rho}.
\end{align}
For  simplicity,  denote
$$
h(s, x, y):=\phi_r(y-x)-\mathbb{E}\big[\phi_r(y-\tilde{X}^{\nu}_s)\big] .
$$
Then  using the H\"older inequality, we rewrite $J_2$ as a double integral by
\begin{align}\label{e2.9}
J_2&=C \int_{|y| \leq R}|y|^2 \mathbb{E}\Big|\frac{1}{\lfloor t\rfloor_{\tau}+1}\sum_{i=0}^{\lfloor t\rfloor_{\tau}}h(i\tau, \tilde{X}^{\nu}_{i\tau},y)\Big| \nn\ \mathrm{d} y
\\ &\leq C \int_{|y| \leq R}|y|^2 \Big(\mathbb{E}\Big|\frac{1}{\lfloor t\rfloor_{\tau}+1}\sum_{i=0}^{\lfloor t\rfloor_{\tau}}h(i\tau, \tilde{X}^{\nu}_{i\tau}, y)\Big|^2\Big)^{\frac{1}{2}} \mathrm{d} y\nn\\
 &\leq \frac{C}{(\lfloor t\rfloor_{\tau}+1)}\int_{|y| \leq R}|y|^2 \mathbb{E}\Bigg[\sum_{i=0}^{\lfloor t\rfloor_{\tau}}\sum_{j=0}^{\lfloor t\rfloor_{\tau}}h(i\tau, \tilde{X}^{\nu}_{i\tau}, y) h(j\tau, \tilde{X}^{\nu}_{j\tau}, y) \Bigg]^{\frac{1}{2}} \mathrm{~d} y \nn\
\\&\leq  \frac{C}{\lfloor t\rfloor_{\tau}+1} \int_{|y| \leq R}|y|^2\left[J_{21}+J_{22}\right]^{\frac{1}{2}} \mathrm{d}y,
\end{align}
where
\begin{align*}
J_{21}=2\sum_{i=0}^{\lfloor t \rfloor_{\tau}}\sum_{j=i+1}^{\lfloor t\rfloor_{\tau}}\mathbb{E}\left[h(i\tau, \tilde{X}^{\nu}_{i\tau}, y)h(j\tau,\tilde{X}^{\nu}_{j\tau},y)\right],~~~J_{22}=\sum_{i=0}^{\lfloor t \rfloor_{\tau}}\mathbb{E}\left[h^2(i\tau, \tilde{X}^{\nu}_{i\tau},y)\right].
\end{align*}
By the  time-homogeneous Markov property of $\tilde{X}$ and the property of conditional expectation, we derive that
\begin{align}
J_{21}=&2\sum_{i=0}^{\lfloor t \rfloor_{\tau}}\sum_{j=i+1}^{\lfloor t\rfloor_{\tau}}\mathbb{E}\left[h(i\tau, \tilde{X}^{\nu}_{i\tau}, y)\mathbb{E}\left(h(j\tau,\tilde{X}^{\nu}_{j\tau},y)\big|\tilde{X}^{\nu}_{i\tau}\right)\right]\nn\
\\ \leq &2\sum_{i=0}^{\lfloor t \rfloor_{\tau}}\sum_{j=i+1}^{\lfloor t\rfloor_{\tau}}\mathbb{E}\left[h(i\tau, \tilde{X}^{\nu}_{i\tau}, y)\mathbb{P}_{(j-i)\tau}\left(h(j\tau,z,y)\right)\big|_{z=\tilde{X}^{\nu}_{i\tau}}\right]\nn\
\\ \leq& 2\|h\|_{\infty}\sum_{i=0}^{\lfloor t \rfloor_{\tau}}\sum_{j=i+1}^{\lfloor t\rfloor_{\tau}}\mathbb{E}\left[\mathbb{P}_{(j-i)\tau}\left(h(j\tau,z,y)\right)\big|_{z=\tilde{X}^{\nu}_{i\tau}}\right]\nn\
\\ \leq& Cr^{-d}\sum_{i=0}^{\lfloor t \rfloor_{\tau}}\sum_{j=i+1}^{\lfloor t\rfloor_{\tau}}\left(\mathbb{E}\left|\mathbb{P}_{(j-i)\tau}\left(h(j\tau,z,y)\right)\big|_{z=\tilde{X}^{\nu}_{i\tau}}\right|^2\right)^{\frac{1}{2}}.\nn\
\end{align}
Owing to $\mathbb{E}\big[\mathbb{P}_{(j-i)\tau}(h(j\tau,z,y))\big|_{z=\tilde{X}^{\nu}_{i\tau}}\big]=0$, applying the weak version of Poincar\'e's inequality \cite[Lemma 3.5]{MR4580925}, we infer that
\begin{align}
J_{21}&\leq Cr^{-d}\sum_{i=0}^{\lfloor t \rfloor_{\tau}}\sum_{j=i+1}^{\lfloor t\rfloor_{\tau}}Var^{\frac{1}{2}}(\mathcal{L}^{\tilde{X}^{\nu}}_{i\tau})\Big\|\nabla_{z} \Big [\mathbb{P}_{(j-i)\tau}(\phi(y-z))\Big]\Big\|_{\infty}\nn\
\\&\leq Cr^{-d}\sum_{i=0}^{\lfloor t \rfloor_{\tau}}\sum_{j=i+1}^{\lfloor t\rfloor_{\tau}}Var^{\frac{1}{2}}(\mathcal{L}^{\tilde{X}^{\nu}}_{i\tau})\sup_{z\in \RR^{d}}\limsup_{h\rightarrow0}\frac{\big|\mathbb{E}\phi_{r}\big(y-\tilde{X}^{z+h}_{(j-i)\tau}\big)-\mathbb{E}\phi_{r}\big(y-\tilde{X}^{z}_{(j-i)\tau}\big)\big|}{|h|}\nn\
\\&\leq Cr^{-d}\sum_{i=0}^{\lfloor t \rfloor_{\tau}}\sum_{j=i+1}^{\lfloor t\rfloor_{\tau}}\|\nabla \phi\|_{\infty}Var^{\frac{1}{2}}(\mathcal{L}^{\tilde{X}^{\nu}}_{i\tau})\sup_{z\in \RR^{d}}\limsup_{h\rightarrow0}\frac{\mathbb{E}\big|\tilde{X}^{z+h}_{(j-i)\tau}-\tilde{X}^{z}_{(j-i)\tau}\big|}{|h|}\nn\
\\&\leq C_{\nu}r^{-d}\|\nabla \phi\|_{\infty}\sum_{i=1}^{\lfloor t \rfloor_{\tau}}\sum_{j=i+1}^{\lfloor t\rfloor_{\tau}}e^{-\frac{\bar{\kappa}_1(j-i)\tau}{2}}\leq C_{\nu}r^{-2d}\sum_{i=0}^{\lfloor t\rfloor_{\tau}}\frac{e^{-\frac{\bar{\kappa}_1\tau}{2}}}{1-e^{-\frac{\bar{\kappa}_1\tau}{2}}}\nn\
\\ &\leq  C_{\nu}r^{-2d}\sum_{i=0}^{\lfloor t\rfloor_{\tau}}\frac{2}{\bar{\kappa}_1\tau}\leq C_{\nu}r^{-2d}\frac{\lfloor t\rfloor_{\tau}+1}{\tau}.\nn\
\end{align}
Similarly, utilizing the weak version of Poincar\'e's inequality \cite[Lemma 3.5]{MR4580925} again implies that
\begin{align*}
J_{22}&=\sum_{i=0}^{\lfloor t \rfloor_{\tau}}\mathcal{L}^{\tilde{X}^{\nu}}_{i\tau}\left(h^2(i\tau, z,y)\right)\leq C\sum_{i=0}^{\lfloor t\rfloor_{\tau}}Var(\mathcal{L}^{\tilde{X}^{\nu}}_{i\tau})\|\nabla [\phi_{r}(y-z)]\|^2_{\infty}\nn\
\\&\leq C_{\nu}(\lfloor t\rfloor_{\tau}+1)r^{-2d}.
\end{align*}
Inserting the above two inequalities into \eqref{e2.9} shows that
\begin{align}\label{e2.10}
J_2\leq \frac{C_{\nu}R^{d+2}r^{-d}}{\sqrt{t}}+\frac{C_{\nu}R^{d+2}r^{-d}}{\sqrt{\lfloor t\rfloor_{\tau}+1}}.
\end{align}
Combining \eqref{e2.7}, \eqref{e2.8} and \eqref{e2.10} yields that
\begin{align}
\mathbb{E}\Big[\mathcal{W}^2_{2}\big(\mathcal{E}^{\tilde{X}^{\nu}}_{t,\tau,r},\mathcal{L}^{\tilde{X}^{\nu}}_{t,\tau,r}\big)\Big]\leq C r^2+C_{\nu}(R-r)^{-\rho}+ \frac{C_{\nu}R^{d+2}r^{-d}}{\sqrt{t}}+\frac{C_{\nu}R^{d+2}r^{-d}}{\sqrt{\lfloor t\rfloor_{\tau}+1}}.
\end{align}
This, together with \eqref{e2.4}-\eqref{e2.6}, leads to that
\begin{align}
\mathbb{E}\Big[\mathcal{W}^2_2\big(\mathcal{E}^{\tilde{X}^{\nu}}_{\tau,t}, \mu^*\big)\Big]&\leq Cr^2+C_{\nu}\big(\tau+\frac{1}{\kappa_1}\big)\frac{1}{t}+C_{\nu}(R-r)^{-\rho}\nn\
\\&~~~+ \frac{C_{\nu}R^{d+2}r^{-d}}{\sqrt{t}}+\frac{C_{\nu}\tau R^{d+2} r^{-d}}{\sqrt{t}}.
\end{align}
Now taking
$$
r=t^{-\frac{\eta}{2}}  \quad \text { and } \quad R=t^{\frac{1}{(d+2)(\rho+2)} },
$$
implies that for each $\nu \in \mathscr{P}_{2+\rho}(\RR^d)$,
$$
\mathbb{E}\Big[\mathcal{W}^2_2\big(\mathcal{E}^{\tilde{X}^{\nu}}_{\tau,t}, \mu^*\big)\Big] \leq C_{\nu} \left(1+\tau\right) t^{-\eta}.
$$
The proof is complete.
\end{proof}
\begin{thm}\label{T3.1}
Under Assumptions \ref{as2},  for any initial distribution $\nu\in \mathscr{P}_{2+\rho}(\RR^{d})$ and $\delta\in (0,1-\bar{\kappa}_2/\bar{\kappa}_1)$, there exists a constant $C_{\nu,\delta}$ such that
\begin{align*}
\mathbb{E}\left[\mathcal{W}^2_{2}\left(\mathcal{E}^{Y^{\nu}}_{\tau,t},\mu^{*}\right)\right]\leq C_{\nu,\delta}\left(1+\tau\right) t^{-\eta\wedge\delta}.
\end{align*}
\end{thm}
\begin{proof}
Applying the elementary inequality and Lemma \ref{L5.1}, it remains estimate that
\begin{align*}
\mathbb{E}\Big[\mathcal{W}_{2}^{2}\big(\mathcal{E}^{Y^{\nu}}_{\tau,t},\mathcal{E}^{\tilde{X}^{\nu}}_{\tau,t}\big)\Big]\leq \int_{0}^{t}\mathbb{E}|Y^{\nu}_{ts}-\tilde{X}^{\nu}_{ts}|^{2}\pi_{\tau,t}(\mathrm{d}s) 
\end{align*}
for the desired result.
Utilizing the It\^o formula  yields that
\begin{align*}
\mathbb{E}|Y^{\nu}_t-\tilde{X}^{\nu}_t|^2&=\int_{0}^{t}\mathbb{E}\Big[2(Y^{\nu}_{s}-X^{\nu}_{s})^{T}\big(f(Y^{\nu}_{s},\mathcal{E}^{Y^{\nu}}_{\tau,s})-f(\tilde{X}^{\nu}_{s},\mu^*)\big)+\big|g(Y^{\nu}_{s},\mathcal{E}^{Y^{\nu}}_{\tau,s})-g(\tilde{X}^{\nu}_{s},\mu^{*})\big|^2\Big]\mathrm{d}s.
\end{align*}
Then by Assumption \ref{as2} we obtain that
\begin{align}\label{J3.13}
\frac{\mathrm{d}}{\mathrm{d}t} \mathbb{E}\big|Y^{\nu}_t-\tilde{X}^{\nu}_t\big|^2&= \mathbb{E}\Big[2(Y^{\nu}_{t}-X^{\nu}_{t})^{T}\big(f(Y^{\nu}_{t},\mathcal{E}^{Y^{\nu}}_{\tau,t})-f(\tilde{X}^{\nu}_{t},\mu^*)\big)+\big|g(Y^{\nu}_{t},\mathcal{E}^{Y^{\nu}}_{\tau,t})-g(\tilde{X}^{\nu}_{t},\mu^{*})\big|^2\Big]\nn\
\\&\leq-\bar{\kappa}_1 \mathbb{E}\big|Y^{\nu}_t-\tilde{X}^{\nu}_t\big|^2+\bar{\kappa}_2 \mathbb{E}\big[\mathcal{W}^2_2\left(\mathcal{E}^{Y^{\nu}}_{\tau,t}, \mu^*\right)\big].
\end{align}
For any $\delta\in (0,1-\bar{\kappa}_2/\bar{\kappa}_1)$, we can choose a constant $\tilde{\kappa}_2\in \mathcal{A}_{\delta}$. Then we have  $\bar{\kappa}_2< \tilde{\kappa}_2<\bar{\kappa}_1$ and
\begin{align}\label{eq3.13}
\limsup_{t\rightarrow\infty}\int_{0}^{1}s^{-\delta}\pi_{\tau,t}(\mathrm{d}s)< \frac{\bar{\kappa}_1}{\tilde{\kappa}_2}.
\end{align}
Using the Young inequality, we obtain that
$$\bar{\kappa}_2 \mathbb{E}\Big[\mathcal{W}^2_2\big(\mathcal{E}^{Y^{\nu}}_{\tau,t}, \mu^*\big)\Big]\leq \tilde{\kappa}_{2} \mathbb{E}\Big[\mathcal{W}^2_2\big(\mathcal{E}^{Y^{\nu}}_{\tau,t}, \mathcal{E}^{\tilde{X}^{\nu}}_{\tau,t}\big)\Big]+ \frac{\bar{\kappa}_2\tilde{\kappa}_2}{\tilde{\kappa}_2-\bar{\kappa}_2} \mathbb{E}\Big[\mathcal{W}^2_2\big(\mathcal{E}^{\tilde{X}^{\nu}}_{\tau,t}, \mu^*\big)\Big].$$
Then inserting the above inequality into \eqref{J3.13} shows that
\begin{align}\label{e3.14}
\frac{\mathrm{d}}{\mathrm{d}t} \mathbb{E}\big|Y^{\nu}_t-\tilde{X}^{\nu}_t\big|^2& \leq -\bar{\kappa}_1 \mathbb{E}\big|Y^{\nu}_t-\tilde{X}^{\nu}_t\big|^2+\tilde{\kappa}_2 \mathbb{E}\Big[\mathcal{W}^2_2\big(\mathcal{E}^{Y^{\nu}}_{\tau,t}, \mathcal{E}^{\tilde{X}}_{\tau,t}\big)\Big]+C_{\delta} \mathbb{E}\Big[\mathcal{W}^2_2\big(\mathcal{E}^{\tilde{X}^{\nu}}_{\tau,t}, \mu^*\big)\Big] \nn\\
&\leq  -\bar{\kappa}_1 \mathbb{E}\big|Y^{\nu}_t-\tilde{X}^{\nu}_t\big|^2+\frac{\tilde{\kappa}_2}{\lfloor t\rfloor_{\tau}+1}\sum_{i=1}^{\lfloor t\rfloor_{\tau}}\mathbb{E}\big|Y^{\nu}_{i\tau}-\tilde{X}^{\nu}_{i\tau}\big|^2+C_{\delta}\mathbb{E}\Big[\mathcal{W}^2_2\big(\mathcal{E}^{\tilde{X}^{\nu}}_{\tau,t}, \mu^*\big)\Big]\nn\
\\&\leq -\bar{\kappa}_1 \mathbb{E}|Y^{\nu}_{t}-\tilde{X}^{\nu}_{t}|^2+\tilde{\kappa}_2\int_{0}^{1}\mathbb{E}|Y^{\nu}_{ts}-\tilde{X}^{\nu}_{ts}|^2\pi_{\tau,t}(\mathrm{d}s)+C_{\nu,\delta} \left(1+\tau\right) t^{-\eta}.
\end{align}
This, along with Lemma \ref{L5}, gives that there is a constant $C_{\nu,\delta}>0$ such that
\begin{align}
\mathbb{E}|Y^{\nu}_t-\tilde{X}^{\nu}_t|^2&\leq C_{\nu,\delta} \left(1+\tau\right) t^{-\eta \wedge \delta}.
\end{align}
Then we obtain that
\begin{align*}
\mathbb{E}\Big[\mathcal{W}_{2}^{2}\big(\mathcal{E}^{Y^{\nu}}_{\tau,t},\mathcal{E}^{\tilde{X}^{\nu}}_{\tau,t}\big)\Big]&\leq \int_{0}^{t}\mathbb{E}|Y^{\nu}_{ts}-\tilde{X}^{\nu}_{ts}|^2\pi_{\tau,t}(\mathrm{d}s)\nn\
\\&\leq C_{\nu,\delta} (1+\tau) t^{-\eta\wedge\delta}\int_{0}^{1}s^{-\eta\wedge\delta}\pi_{\tau,t}(\mathrm{d}s)\nn\
\\&\leq C_{\nu,\delta} (1+\tau) t^{-\eta\wedge\delta}.
\end{align*}
Combining this and Lemma \ref{L2.4} yields the desired result.
The proof is complete.
\end{proof}
\subsection{The EM scheme  for self-interacting process}
    This subsection will  design the EM schemes for the self-interacting process, and  study the convergence rates of the EM schemes in  the empirical approximations to  the invariant probability measure $\mu^*$ of MV-SDE \eqref{eq1} in  $\mathcal{W}_2$-Wasserstein distance.   For this, we assume an additional condition.
\begin{assp} \label{as3}
There exists a constant $L>0$ such that for any $x,y\in \RR^{d}$  and $\mu\in \mathscr{P}_{2+\rho}(\RR^{d})$,
\begin{align*}
&|f(x,\mu)-f(y,\mu)|\vee |g(x,\mu)-g(y,\mu)|\leq L|x-y|, \\
&|g(x,\mu)|^2\leq L(1+|x|^{2}+\mu(|\cdot|^2)).
\end{align*}
\end{assp}
 For any fixed $\tau>0$, we may assume that there is a sufficiently large $M\in \mathbb{N}_{+}$ such that the
step size $\Delta=\tau/M\in (0,1]$. Then for any $k\in \mathbb{N}$, define
\begin{equation}
\begin{cases}
Z_0=X_0,\\
\mathcal{E}^{Z}_{\tau, k\tau}=\frac{1}{k+1}\sum_{i=0}^{k}\boldsymbol{\delta}_{Z_{i\tau}},\\
Z_{k\tau+(m+1)\Delta}=Z_{k\tau+m\Delta}+f\big(Z_{k\tau+m\Delta}, \mathcal{E}^{Z}_{\tau,k\tau}\big)\Delta
+g\big(Z_{k\tau+m\Delta}, \mathcal{E}^{Z}_{\tau, k\tau}\big)\sqrt{\Delta} \xi_{k,m},
\\~~~~~~~~~~~~~~~~~~ ~~~m=0,1,\cdots,M-1,
\end{cases}
\end{equation}
where $\{\xi_{k,m}\}_{k,m\geq0}$ is  the sequence of i.i.d. Gaussian random variables with mean 0 and variance 1. Based on the above scheme, define two versions of numerical solutions  by
\begin{align}\label{e18}
Z_{t}=Z_{\lfloor t\rfloor_{\Delta}\Delta},~~t\geq0,
\end{align}
and
\begin{align}\label{e4.3}
\bar{Z}_t&=Z_0+\int_{0}^{t}f(Z_s, \mathcal{E}^{Z}_{\tau,s})\mathrm{d}s+\int_{0}^{t}g(Z_s, \mathcal{E}^{Z}_{\tau,s})\mathrm{d}B_{s},~~t\geq0.
\end{align}
One observes that for any $s>0$,  there is a $k\in \mathbb{N}$ such that $s\in [k\tau, (k+1)\tau)$ and  $\mathcal{E}_{\tau,s}^{Z}=\mathcal{E}_{\tau,k\tau}^Z$. Thus, we further derive that
processes $\bar{Z}$ and $Z$  exhibit equality at the discrete-time nodes, i.e., $\bar{Z}_{k\Delta}=Z_{k\Delta}$ for all $k\in \mathbb{N}$. Obviously, $\bar{Z}_{k\tau}=Z_{k\tau}$ for any $k\in \mathbb{N}$. Thus,  $$\mathcal{E}^{\bar{Z}}_{\tau,t}=\frac{1}{\lfloor t\rfloor_{\tau}+1}\sum_{k=0}^{\lfloor t\rfloor_{\tau}}\boldsymbol{\delta}_{\bar{Z}_{k\tau}}=\frac{1}{\lfloor t\rfloor_{\tau}+1}\sum_{k=0}^{\lfloor t\rfloor_{\tau}}\boldsymbol{\delta}_{Z_{k\tau}}=\mathcal{E}^{Z}_{\tau,t}.$$  Thus, it follows from \eqref{e4.3} that
 \begin{align}\label{e4.4}
 \bar{Z}_{t}=X_0+\int_{0}^{t}f(Z_s, \mathcal{E}^{\bar{Z}}_{\tau,s})\mathrm{d}s+\int_{0}^{t}g(Z_s,\mathcal{E}^{\bar{Z}}_{\tau,s})\mathrm{d}B_s,~~X_0\sim\nu.
 \end{align}

Since the uniform moment boundedness of exact and numerical solutions in infinite time intervals is closely related to the uniform convergence of EM scheme in infinite time interval, we will study the uniform moment boundedness results of exact and numerical solutions, respectively. 
\begin{lem}
Let Assumptions \ref{as2} hold. Then for any initial distribution $\nu\in \mathscr{P}_{2+\rho}(\RR^{d})$, there exists a constant $C_{\nu}$ such that
\begin{align*}
\sup_{t\geq 0}\mathbb{E}|Y^{\nu}_{t}|^{2}\leq C_{\nu}.
\end{align*}
\end{lem}
\begin{proof}
Applying the It\^o formula yields that for any $t>0$,
\begin{align}
\mathbb{E}|Y^{\nu}_{t}|^{2}&=\nu(|\cdot|^2)+\int_{0}^{t}\mathbb{E}\Big[2(Y^{\nu}_{s})^{T}f(Y^{\nu}_{s},\mathcal{E}^{Y^{\nu}}_{\tau,s})+|g(Y^{\nu}_{s}, \mathcal{E}^{Y^{\nu}}_{\tau,s})|^{2}\Big]\mathrm{d}s.\nn\
\end{align}
 Thus,  by \eqref{e2.1} and \eqref{eq2.9} we deduce that
\begin{align*}
\frac{\mathrm{d}\mathbb{E}|Y^{\nu}_{t}|^{2}}{\mathrm{d}t}&\leq  -\kappa_1\mathbb{E}|Y^{\nu}_{t}|^2+\kappa_2\mathbb{E}\left[ \mathcal{E}^{Y^{\nu}}_{\tau,t}(|\cdot|^2)\right]+C\nn\
\\&\leq -\kappa_1\mathbb{E}|Y^{\nu}_{t}|^2+\kappa_2 \int_{0}^{1}\mathbb{E}|Y^{\nu}_{ts}|^{2}\pi_{\tau,t}(\mathrm{d}s)+C.
\end{align*}
Therefore, it follows from Lemma \ref{CLem2.5} that
\begin{align*}
\mathbb{E}|Y^{\nu}_{t}|^{2}\leq \frac{\kappa_2\nu(|\cdot|^2)+C}{\kappa_1-\kappa_2}\vee \nu(|\cdot|^2),
\end{align*}
 This implies the desired result.
\end{proof}
\begin{lem}\label{L7}
Let Assumptions \ref{as2} and \ref{as3} hold. Then there exists a constant $\Delta^*\in (0,1]$ such that for any $\Delta\in (0,\Delta^*]$ and the initial distribution $\nu\in \mathscr{P}_{2+\rho}(\RR^{d})$, the numerical solution $\bar{Z}$ defined by \eqref{e18} and \eqref{e19}  satisfies that
\begin{align*}
\sup_{t\geq0}\mathbb{E}|\bar{Z}^{\nu}_{t}|^{2}\leq C_{\nu}.
\end{align*}
\end{lem}
\begin{proof}
Using the It\^o formula for \eqref{e4.3} yields that for any $t>0$,
\begin{align}\label{e2.15}
\frac{\mathrm{d}\mathbb{E}|\bar{Z}^{\nu}_{t}|^{2}}{\mathrm{d}t}=\mathbb{E}\Big[\Big(2(\bar{Z}^{\nu}_{t})^{T}f(Z^{\nu}_{t}, \mathcal{E}^{\bar{Z}^{\nu}}_{\tau,t})+\big|g(Z^{\nu}_{t},\mathcal{E}^{\bar{Z}^{\nu}}_{\tau,t})\big|^2\Big)\Big].
\end{align}
Utilizing the Young inequality shows that for any $\epsilon\in (0,(\kappa_1-\kappa_2)/2)$,
\begin{align}
&2(\bar{Z}^{\nu}_{t})^{T}f(Z^{\nu}_{t}, \mathcal{E}^{\bar{Z}^{\nu}}_{\tau,t})+\big|g(Z^{\nu}_{t},\mathcal{E}^{\bar{Z}^{\nu}}_{\tau,t})\big|^2\nn\
\\ \leq & 2(\bar{Z}^{\nu}_{t})^{T}f(\bar{Z}^{\nu}_{t}, \mathcal{E}^{\bar{Z}^{\nu}}_{\tau,t})+2(\bar{Z}^{\nu}_{t})^{T}\left(f(Z^{\nu}_{t},\mathcal{E}^{\bar{Z}^{\nu}}_{\tau,t})-f(\bar{Z}^{\nu}_t, \mathcal{E}^{\bar{Z}^{\nu}}_{\tau,t})\right)\nn\
\\&~~~+(1+\rho)\big|g(\bar{Z}^{\nu}_{t},\mathcal{E}^{\bar{Z}^{\nu}}_{\tau,t})\big|^2+\frac{1+\rho}{\rho}\big|g(Z^{\nu}_{t},\mathcal{E}^{\bar{Z}^{\nu}}_{\tau,t})-g(\bar{Z}^{\nu}_t, \mathcal{E}^{\bar{Z}^{\nu}}_{\tau,t})\big|^{2}\nn\
\\ \leq&2(\bar{Z}^{\nu}_{t})^{T}f(\bar{Z}^{\nu}_{t}, \mathcal{E}^{\bar{Z}^{\nu}}_{\tau,t})+(1+\rho)\big|g(\bar{Z}^{\nu}_{t},\mathcal{E}^{\bar{Z}^{\nu}}_{\tau,t})\big|^2+\epsilon|\bar{Z}^{\nu}_{t}|^2\nn\
\\& ~~~+C\Big(\big|f(Z^{\nu}_{t},\mathcal{E}^{\bar{Z}^{\nu}}_{\tau,t})-f(\bar{Z}^{\nu}_t, \mathcal{E}^{\bar{Z}^{\nu}}_{\tau,t})\big|^2+\big|g(Z^{\nu}_{t},\mathcal{E}^{\bar{Z}^{\nu}}_{\tau,t})-g(\bar{Z}^{\nu}_t, \mathcal{E}^{\bar{Z}^{\nu}}_{\tau,t})\big|^{2}\Big).
\end{align}
Then by  Assumptions \ref{as2} and \ref{as3}, we have
\begin{align}\label{e16}
&2(\bar{Z}^{\nu}_{t})^{T}f(Z^{\nu}_{t}, \mathcal{E}^{\bar{Z}^{\nu}}_{\tau,t})+(1+\rho)\big|g(Z^{\nu}_{t},\mathcal{E}^{\bar{Z}^{\nu}}_{\tau,t})\big|^2\nn\
\\ \leq & -(\kappa_1-\epsilon)|\bar{Z}^{\nu}_{t}|^{2}+\kappa_2\mathcal{E}^{\bar{Z}^{\nu}}_{\tau,t}(|\cdot|^2)+C|\bar{Z}^{\nu}_t-Z^{\nu}_t|^{2}+C.
\end{align}
Inserting the above inequality into \eqref{e2.15} and using \eqref{e2.1}, we arrive at that
\begin{align}\label{e4.8}
\frac{\mathrm{d}\mathbb{E}|\bar{Z}^{\nu}_{t}|^{2}}{\mathrm{d}t}=-(\kappa_1-\epsilon)\mathbb{E}|\bar{Z}^{\nu}_{t}|^2+\kappa_2\int_{0}^{1}\mathbb{E}|\bar{Z}^{\nu}_{ts}|^2\pi_{\tau,t}(\mathrm{d}s)+C\mathbb{E}|\bar{Z}^{\nu}_{t}-Z^{\nu}_{t}|^2\mathrm{d}s+C.
\end{align}
On the other hand, for any $t>0$, there are $k,m\in \mathbb{N}$ and $0\leq m\leq M-1$ such that $t\in[k\tau+m\Delta, k\tau+(m+1)\Delta)$. Then, we have
$$Z^{\nu}_{t}=Z^{\nu}_{k\tau+m\Delta},~~~\mathcal{E}^{\bar{Z}^{\nu}}_{\tau,t}=\mathcal{E}^{\bar{Z}^{\nu}}_{\tau,k\tau}, ~~a.s.$$
Hence, by Assumption \ref{as3} we derive from \eqref{e2.15} that
\begin{align*}
\mathbb{E}|\bar{Z}^{\nu}_{t}-Z^{\nu}_{t}|^2&\leq \mathbb{E}\big|f\big(Z^{\nu}_{k\tau}+m\Delta, \mathcal{E}^{\bar{Z}^{\nu}}_{\tau,k\tau}\big)\big|^{2}\Delta^2+\mathbb{E}\big|g\big(Z^{\nu}_{k\tau+m\Delta},\mathcal{E}^{\bar{Z}^{\nu}}_{\tau,k\tau}\big)\big|^{2}\Delta\nn\
\\&\leq C\Delta \mathbb{E}\left[1+|Z^{\nu}_{k\tau+m\Delta}|^2+\mathcal{E}^{\bar{Z}^{\nu}}_{\tau,k\tau}(|\cdot|^2)\right]\nn\
\\&\leq C\Delta \mathbb{E}\left[1+|Z^{\nu}_{t}|^2+\mathcal{E}^{\bar{Z}^{\nu}}_{\tau,t}(|\cdot|^2)\right].\nn\
\end{align*}
Furthermore, using  \eqref{eq2.9}  yields that
\begin{align*}
\mathbb{E}|\bar{Z}^{\nu}_{t}-Z^{\nu}_{t}|^2&\leq C\Delta \left[1+ \mathbb{E}|Z^{\nu}_{t}|^2+\int_{0}^{1} \mathbb{E}|\bar{Z}^{\nu}_{ts}|^{2}\pi_{\tau,t}(\mathrm{d}s)\right]\nn\
\\&\leq  C\Delta \left[1+ \mathbb{E}|\bar{Z}^{\nu}_{t}-Z^{\nu}_t|^2+ \mathbb{E}|\bar{Z}^{\nu}_t|^2+\int_{0}^{1} \mathbb{E}|\bar{Z}^{\nu}_{ts}|^{2}\pi_{\tau,t}(\mathrm{d}s)\right].
\end{align*}
Choose $\Delta_1^*$ small enough such that $C\Delta_1^*\leq 1/2$. Therefore, for any $\Delta\in (0,\Delta^*]$,
\begin{align}\label{4.9}
\mathbb{E}|\bar{Z}^{\nu}_{t}-Z^{\nu}_{t}|^2\leq C\Delta+C\Delta  \mathbb{E}|\bar{Z}^{\nu}_{t}|^2+C\Delta\int_{0}^{1} \mathbb{E}|\bar{Z}^{\nu}_{ts}|^2\pi_{\tau,t}(\mathrm{d}s).
\end{align}
Inserting the above inequality into \eqref{e4.8} shows that
\begin{align}\label{e4.8}
\frac{\mathrm{d}\mathbb{E}|\bar{Z}^{\nu}_{t}|^{2}}{\mathrm{d}t}=-(\kappa_1-\epsilon-C\Delta)\mathbb{E}|\bar{Z}^{\nu}_{t}|^2+(\kappa_2+C\Delta)\int_{0}^{1}\mathbb{E}|\bar{Z}^{\nu}_{ts}|^2\pi_{\tau,t}(\mathrm{d}s)+C.
\end{align}
Thanks to $\kappa_2<\kappa_1-\varepsilon$,  we  further determine  a constant $\Delta^*\in (0,\Delta^*_1]$ such that $C\Delta^*<(\kappa_1-\kappa_2-\epsilon)/2$. Thus, for any $\Delta\in (0,\Delta^*]$, we have $\kappa_2+C\Delta<\kappa_1-\epsilon-C\Delta$. 
Then applying  Lemma \ref{CLem2.5} yields that
\begin{align*}
\mathbb{E}|Z^{\nu}_{t}|^2\leq \frac{4C+4\beta \nu(|\cdot|^2)}{\kappa_1-\kappa_2}\vee \nu(|\cdot|^2),
\end{align*}
which implies the desired result.
\end{proof}
\begin{coro}\label{Cor4.1}
Let Assumptions \ref{as2} and \ref{as3} hold. Then for any initial distribution $\nu\in \mathscr{P}_{2+\rho}(\RR^{d})$  and $\Delta\in (0,\Delta^*]$, there exists a constant $C_{\nu}$ such that  numerical solution processes $\bar{Z}$ and $Z$ satisfy that
\begin{align*}
\sup_{t\geq0}\mathbb{E}|\bar{Z}^{\nu}_{t}-Z^{\nu}_{t}|^{2}\leq C_{\nu}\Delta.
\end{align*}
\end{coro}
\begin{proof}
Combining \eqref{4.9} and Lemma \ref{L7}, we deduce  that for any $t>0$ and $\Delta\in (0,\Delta^*]$,
\begin{align*}
\sup_{t\geq0}\mathbb{E}|\bar{Z}^{\nu}_{t}-Z^{\nu}_{t}|^2\leq C\Delta+C\Delta \sup_{t\geq0}\mathbb{E}|\bar{Z}^{\nu}_{t}|^2+C\Delta\sup_{t\geq0}\mathbb{E}|\bar{Z}^{\nu}_{t}|^2\int_{0}^{1}  \pi_{\tau,t}(\mathrm{d}s)\leq C_{\nu}\Delta.
\end{align*}
The proof is complete.
\end{proof}
\begin{thm}\label{LJP4.1}
Let  Assumptions \ref{as2} and \ref{as3} hold. Then for any initial distribution $\nu\in \mathscr{P}_{2+\rho}(\RR^{d})$, $\Delta\in (0,\Delta^*]$, there exists a constant $C_{\nu}$ such that
\begin{align*}
\sup_{t\geq0}\mathbb{E}|Y^{\nu}_t-Z^{\nu}_t|^2\leq C_{\nu}\Delta.
\end{align*}
\end{thm}
\begin{proof}
It follows from  \eqref{e2.13} and  \eqref{e4.4}~that 
\begin{align}
Y^{\nu}_{t}-\bar{Z}^{\nu}_{t}&=\int_{0}^{t}\left(f(Y^{\nu}_s,\mathcal{E}^{Y^{\nu}}_{\tau,s})-f(Z^{\nu}_{s},\mathcal{E}^{\bar{Z}^{\nu}}_{\tau,s})\right)\mathrm{d}s\nn\
\\&~~~+\int_{0}^{t}\left(g(Y^{\nu}_s,\mathcal{E}^{Y^{\nu}}_{\tau,s})-g(Z^{\nu}_s,\mathcal{E}^{\bar{Z}^{\nu}}_{\tau,s})\right)\mathrm{d}B_s.
\end{align}
Employing the It\^o formula and \eqref{2.10} leads to that
\begin{align}\label{cyp22}
&\frac{\mathrm{d}\mathbb{E}|Y^{\nu}_{t}-\bar{Z}^{\nu}_{t}|^{2}}{\mathrm{d}t}\nn\
\\=&\mathbb{E}\Big[2\big(Y^ {\nu}_{t}-\bar{Z}^{\nu}_{t}\big)^{T}\big(f(Y^{\nu}_{t},\mathcal{E}^{Y^{\nu}}_{\tau,t})-f(Z^{\nu}_{t},\mathcal{E}^{\bar{Z}^{\nu}}_{\tau,t})\big)+\big|g(Y^{\nu}_t,\mathcal{E}^{Y^{\nu}}_{\tau,t})-g(Z^{\nu}_t,\mathcal{E}^{\bar{Z}^{\nu}}_{\tau,t})\big|^2\Big]\nn\
\\ \leq & \mathbb{E}\Big[2\big(Y^{\nu}_{t}-\bar{Z}^{\nu}_{t}\big)^{T}\big(f(Y^{\nu}_{t},\mathcal{E}^{Y^{\nu}}_{\tau,t})-f(\bar{Z}^{\nu}_{t},\mathcal{E}^{\bar{Z}^{\nu}}_{\tau,t})\big)+\big|g(Y^{\nu}_t,\mathcal{E}^{Y^{\nu}}_{\tau,t})-g(\bar{Z}^{\nu}_t,\mathcal{E}^{\bar{Z}^{\nu}}_{\tau,t})\big|^2\Big]+\mathcal{J}_1+\mathcal{J}_2\nn\
\\\leq  &\Big(-\bar{\kappa}_1 \mathbb{E}|Y^{\nu}_{t}-\bar{Z}^{\nu}_{t}|^{2}+\bar{\kappa}_2\int_{0}^{1}\mathbb{E}|Y^{\nu}_{ts}-\bar{Z}^{\nu}_{ts}|^{2}\pi_{\tau,t}(\mathrm{d}s)\Big)+\mathcal{J}_1+\mathcal{J}_2,
\end{align}
where
\begin{align*}
&\mathcal{J}_1=2\mathbb{E}\Big[\big(Y^{\nu}_{t}-\bar{Z}^{\nu}_{t}\big)^{T}\big(f(\bar{Z}^{\nu}_{t},\mathcal{E}^{\bar{Z}^{\nu}}_{\tau,t})-f(Z^{\nu}_{t},\mathcal{E}^{\bar{Z}^{\nu}}_{\tau,t})\big)\Big]+\mathbb{E}\big|g(\bar{Z}^{\nu}_{t},\mathcal{E}^{\bar{Z}^{\nu}}_{\tau,t})-g(Z^{\nu}_{t},\mathcal{E}^{\bar{Z}^{\nu}}_{\tau,t})\big|^2,\nn\
\\&\mathcal{J}_2=2\mathbb{E}\Big(\big|g(Y^{\nu}_t,\mathcal{E}^{Y^{\nu}}_{\tau,t})-g(\bar{Z}^{\nu}_t,\mathcal{E}^{\bar{Z}^{\nu}}_{\tau,t})\big|\big|g(\bar{Z}^{\nu}_t,\mathcal{E}^{\bar{Z}^{\nu}}_{\tau,t})-g(Z^{\nu}_t,\mathcal{E}^{\bar{Z}^{\nu}}_{\tau,t})\big|\Big).\nn\
\end{align*}
By  Assumption \ref{as3} and the Young inequality, we derive that  for any $\varepsilon>0$
\begin{align*}
\mathcal{J}_1\leq \frac{\varepsilon}{2}\mathbb{E}|Y^{\nu}_t-\bar{Z}^{\nu}_t|^2+C\mathbb{E}|\bar{Z}^{\nu}_{t}-Z^{\nu}_t|^2,\nn\
\end{align*}
and
\begin{align*}
\mathcal{J}_2\leq \frac{\varepsilon }{2} \mathbb{E}|Y^{\nu}_t-\bar{Z}^{\nu}_{t}|^{2}+\varepsilon\int_{0}^{1}\mathbb{E}|Y^{\nu}_{ts}-\bar{Z}^{\nu}_{ts}|^2\pi_{\tau,t}(\mathrm{d}s)+C\mathbb{E}|\bar{Z}^{\nu}_{t}-Z^{\nu}_t|^2.
\end{align*}
Inserting the above inequalities into \eqref{cyp22} and using Corollary \ref{Cor4.1} yield that
\begin{align*}
\frac{\mathrm{d}\mathbb{E}|Y^{\nu}_{t}-\bar{Z}^{\nu}_{t}|^{2}}{\mathrm{d}t}&\leq -\left(\bar{\kappa}_1-\varepsilon\right)\mathbb{E}|Y^{\nu}_t-\bar{Z}^{\nu}_{t}|^{2}+\left(\bar{\kappa}_2+\varepsilon\right)\int_{0}^{1}\mathbb{E}|Y^{\nu}_{ts}-\bar{Z}^{\nu}_{ts}|^{2}\pi_{\tau,t}(\mathrm{d}s)+C_{\nu}\Delta.
\end{align*}
Letting $\varepsilon\in (0,(\bar{\kappa}_1-\bar{\kappa}_2)/4)$, we have $\bar{\kappa}_2+\varepsilon<\bar{\kappa}_1-\varepsilon$. Then applying Lemma \ref{CLem2.5} shows that
\begin{align}
\mathbb{E}|Y^{\nu}_{t}-\bar{Z}^{\nu}_{t}|^{2}\leq \frac{C_{\nu}\Delta}{\bar{\kappa}_1-\bar{\kappa}_2-2\varepsilon}\leq \frac{2C_{\nu}\Delta}{\bar{\kappa}_1-\bar{\kappa}_2}\leq C_{\nu}\Delta,
\end{align}
where $C$ is independent of $t$. This, together with Corollary \ref{Cor4.1}, implies the desired result.
\end{proof}
\begin{thm}\label{L4.2}
Let  Assumptions \ref{as2} and \ref{as3} hold. Then for any initial distribution $\nu\in \mathscr{P}_{2+\rho}(\RR^{d})$, $\Delta\in (0,\Delta^*]$ and $\delta\in (0,1-\bar{\kappa}_2/\bar{\kappa}_1)$, there exists a constant $C_{\nu,\delta}>0$ such that
\begin{align*}
\mathbb{E}\Big[\mathcal{W}_{2}^{2}\big(\mathcal{E}^{Z^{\nu}}_{\tau,t},\mu^*\big)\Big]\leq C_{\nu,\delta} (1+\tau)\big(t^{-\eta\wedge\delta}+\Delta\big).
\end{align*}
\end{thm}
\begin{proof}
It follows from the elementary inequality  that
\begin{align*}
\mathcal{W}_{2}\left(\mathcal{E}^{Z^{\nu}}_{\tau,t},\mu^*\right)\leq \mathcal{W}_{2}(\mathcal{E}^{Z^{\nu}}_{\tau,t}, \mathcal{E}^{Y^{\nu}}_{\tau,t}) +\mathcal{W}_{2}\left(\mathcal{E}^{Y^{\nu}}_{\tau,t},\mu^*\right).
\end{align*}
Applying Theorem \ref{LJP4.1} and \eqref{2.10} implies that there exists a constant $C_{\nu}$ such that
\begin{align*}
\mathbb{E}\left[\mathcal{W}_{2}^{2}\left(\mathcal{E}^{Z^{\nu}}_{\tau,t},\mathcal{E}^{Y^{\nu}}_{\tau,t}\right)\right]&\leq \int_{0}^{1}\mathbb{E}|Z^{\nu}_{ts}-Y^{\nu}_{ts}|^{2}\pi_{\tau,t}(\mathrm{d}s)\leq \sup_{t\geq0}\mathbb{E}|Z^{\nu}_{t}-Y^{\nu}_{t}|^{2} \int_{0}^{1}\pi_{\tau,t}(\mathrm{d}s)\leq C_{\nu}\Delta.
\end{align*}
In view of Theorem \ref{T3.1}, for any $\delta \in (0,1-
\bar{\kappa}_2/\bar{\kappa}_1)$, there exists a constant $C_{\nu,\delta}$ such that
\begin{align*}
\mathbb{E}\Big[\mathcal{W}^2_{2}\big(\mathcal{E}^{Y^{\nu}}_{\tau,t},\mu^{*}\big)\Big]\leq C_{\nu,\delta}\left(1+\tau\right) t^{-\eta\wedge\delta}.
\end{align*}
Combining the above inequalities  implies the desired result.
\end{proof}
\section{EM scheme in the averaged weighted empirical approximation }
 The previous section has established the convergence of the EM scheme in the weighted empirical approximation of the invariant probability measure of MV-SDE.  Thus,  one can use the EM scheme to simulate  only one process path $Y$  to evolve the invariant probability measure of MV-SDE. Borrowing  Monte Carlo idea, we introduce the multi-particle system as
\begin{equation}\label{eq5.1}
\begin{cases}
\mathrm{d}Y^{j,N}_{t}=f\left(Y^{j,N}_{t}, \frac{1}{N}\sum_{j=1}^{N}\mathcal{E}^{Y^{j,N}}_{\tau,t}\right)\mathrm{d}t+g\left(Y^{j,N}_{t},\frac{1}{N}\sum_{j=1}^{N}\mathcal{E}^{Y^{j,N}}_{\tau,t}\right)\mathrm{d}B^{j}_t, ~~1\leq j\leq N,
\\ Y^{j,N}_{0}=X^{j}_{0},
\end{cases}
\end{equation}
on $t\geq0$, where  $(X^{j}_{0}, \{B^{j}_{t}\}_{t\geq0}), 1\leq j\leq N$ are independent copies of $(X_0, \{B_t\}_{t\geq0})$.  It can be observed  from the particle system \eqref{eq5.1} that $Y^{j,N}, 1\leq j\leq N$ are identically distributed. Next  we  will consider the convergence of the EM scheme in the  averaged weighted empirical approximation of the invariant probability measure $\mu^*$.  Along the same lines as in Section 3, this section consists of two main parts. Firstly, establishing the convergence between the averaged weighted empirical measure of the self-interacting process and the invariant probability measure of MV-SDE. Further, constructing the EM sheme for the multi-particle self-interacting process and then proving the convergence between the average weighted empirical measure of the numerical solution and the invariant probability measure of MV-SDE.
\subsection{The averaged weighted empirical approximation}

To estimate the convergence between the averaged weighted empirical measure of the self-interacting process and the invariant probability measure of MV-SDE, we  need to introduce an auxiliary process as
\begin{align}\label{e3.19}
\mathrm{d}\tilde{X}^{j}_{t}=f(\tilde{X}^{j}_{t},\mu^{*})\mathrm{d}t+g(\tilde{X}^{j}_{t},\mu^{*})\mathrm{d}B^{j}_{t},~~~\tilde{X}_0^{j}=X_0^{j},~~\forall 1\leq j\leq N,
\end{align}
on $t\geq0$, where $\mu^{*}$ is the unique invariant probability measure of MV-SDE \eqref{eq1}. It needs to point out that  $\tilde{X}^{j}, 1\leq j\leq N$ are independent and identically distributed (i.i.d.). Similar to the analysis in the  subsection \ref{S3}, we remain to analyze the errors 
$$\mathbb{E}\Big[\mathcal{W}^2_{2}\Big(\frac{1}{N}\sum_{j=1}^{N}\mathcal{E}^{\tilde{X}^{j}}_{\tau,t}, \mu^*\Big)\Big]~~\hbox{and}~~\mathbb{E}\Big[\mathcal{W}^{2}_{2}\Big(\frac{1}{N}\sum_{j=1}^{N}\mathcal{E}^{\tilde{X}^{j}}_{\tau,t}, \frac{1}{N}\sum_{j=1}^{N}\mathcal{E}^{Y^{j,N}}_{\tau,t}\Big)\Big]$$
in turn. To avoid repetition, we outline only the essential proofs below.
\begin{lem}\label{L3.2}
Under Assumption \ref{as2}, for any initial distribution $\nu \in \mathscr{P}_{2+\rho}(\RR^{d})$, there exists a constant $C_{\nu}$ independent of $N$ such that
$$
\mathbb{E}\Big[\mathcal{W}^2_2\Big(\frac{1}{N} \sum_{j=1}^N \mathcal{E}_{\tau,t}^{\tilde{X}^{j,\nu}}, \mu^*\Big)^2\Big] \leq C_{\nu} \left(1+\tau\right) t^{-\eta}N^{-\frac{1}{2}}.
$$
\end{lem}
\begin{proof}
Define
$$
\frac{1}{N}\sum_{j=1}^{N}\mathcal{E}^{\tilde{X}^{j,\nu}}_{\tau,t,r}:=\frac{1}{N}\sum_{j=1}^{N}\Phi_r \times \mathcal{E}^{\tilde{X}^{j,\nu}}_{\tau,t}=\frac{1}{N}\frac{1}{\lfloor t\rfloor_{\tau}+1}\sum_{j=1}^{N}\sum_{i=0}^{\lfloor t\rfloor_{\tau}}\Phi_r\times \boldsymbol{\delta}_{\tilde{X}^{j,\nu}_{i\tau}}
$$
and
$$
\frac{1}{N}\sum_{j=1}^{N}\mathcal{L}^{\tilde{X}^{j,\nu}}_{\tau,t,r}:= \frac{1}{N}\sum_{j=1}^{N}\Phi_r\times \mathcal{L}^{\tilde{X}^{j,\nu}}_{\tau,t}=\frac{1}{N}\frac{1}{\lfloor t\rfloor_{\tau} +1}\sum_{j=1}^{N}\sum_{i=0}^{\lfloor t\rfloor_{\tau}}\Phi_r\times \mathcal{L}^{\tilde{X}^{j,\nu}}_{i\tau},
$$
where $\Phi_r$ is defined in the proof of Lemma \ref{L2.4}.
Employing   Lemma \ref{L3}, we derive that
\begin{align}\label{e2.5}
\mathcal{W}^2_2\Big(\frac{1}{N}\sum_{j=1}^{N}\mathcal{E}^{\tilde{X}^{j,\nu}}_{\tau,t}, \frac{1}{N}\sum_{j=1}^{N}\mathcal{E}^{\tilde{X}^{j,\nu}}_{\tau,t,r}\Big)+\mathcal{W}^2_2\Big(\frac{1}{N}\sum_{j=1}^{N}\mathcal{L}^{\tilde{X}^{j,\nu}}_{\tau,t, r}, \frac{1}{N}\sum_{j=1}^{N}\mathcal{L}^{\tilde{X}^{j,\nu}}_{\tau,t}\Big) \leq \mathbb{E}|r\xi|^2 \leq C r^2~~~a.s.
\end{align}
Thanks to the fact that $\tilde{X}^{j,\nu},~1\leq j\leq N$ are i.i.d, using \eqref{e2.6}  we deduce that
\begin{align}\label{e5.20}
\mathcal{W}^2_2\Big(\frac{1}{N}\sum_{j=1}^{N}\mathcal{L}^{\tilde{X}^{j,\nu}}_{\tau,t}, \mu^*\Big)& \leq \frac{1}{N}\sum_{j=1}^{N}\Big[\frac{1}{\lfloor t\rfloor_{\tau}+1}\sum_{i=0}^{\lfloor t\rfloor_{\tau}}\mathcal{W}^2_2\Big(\mathcal{L}^{\tilde{X}^{j,\nu}}_{i\tau}, \mu^*\Big)\Big] \nn\
\\&\leq C_{\nu}\Big(\tau+\frac{1}{\bar{\kappa}_1}\Big)\frac{1}{t}.
\end{align}
Using the elementary inequality yields that
\begin{align}\label{J3.22}
\mathcal{W}_2\Big(\frac{1}{N}\sum_{j=1}^{N}\mathcal{E}^{\tilde{X}^{j,\nu}}_{\tau,t}, \mu^*\Big) \leq & \mathcal{W}_2\Big(\frac{1}{N}\sum_{j=1}^{N}\mathcal{E}^{\tilde{X}^{j,\nu}}_{\tau,t}, \frac{1}{N}\sum_{j=1}^{N}\mathcal{E}^{\tilde{X}^{j,\nu}}_{\tau,t,r}\Big)+\mathcal{W}_2\Big(\frac{1}{N}\sum_{j=1}^{N}\mathcal{E}^{\tilde{X}^{j,\nu}}_{\tau, t,r}, \frac{1}{N}\sum_{j=1}^{N}\mathcal{L}^{\tilde{X}^{j,\nu}}_{\tau, t,r}\Big)\nn\\
& +\mathcal{W}_2\Big(\frac{1}{N}\sum_{j=1}^{N}\mathcal{L}^{\tilde{X}^{j,\nu}}_{\tau,t, r}, \frac{1}{N}\sum_{j=1}^{N}\mathcal{L}^{\tilde{X}^{j,\nu}}_{\tau,t}\Big)+\mathcal{W}_2\Big(\frac{1}{N}\sum_{j=1}^{N}\mathcal{L}^{\hat{X}^{j,\nu}}_{\tau,t}, \mu^*\Big).
\end{align}
Thus, it remains to estimate 
$$\mathbb{E}\Big[\mathcal{W}^2_2\Big(\frac{1}{N}\sum_{j=1}^{N}\mathcal{E}^{\tilde{X}^{j,\nu}}_{\tau,t,r}, \frac{1}{N}\sum_{j=1}^{N}\mathcal{L}^{\tilde{X}^{j,\nu}}_{\tau,t,r}\Big)\Big].$$ 
Furthermore,  employing the density coupling lemma \cite[Lemma 3.3]{MR4580925} leads to that
\begin{align}\label{e5.5}
& \mathbb{E}\Big[\mathcal{W}^2_2\Big(\frac{1}{N}\sum_{j=1}^{N}\mathcal{E}^{\tilde{X}^{j,\nu}}_{\tau,t,r}, \frac{1}{N}\sum_{j=1}^{N}\mathcal{L}^{\tilde{X}^{j,\nu}}_{\tau,t,r }\Big)\Big] \nn\\
\leq & C \mathbb{E} \int_{\mathbb{R}^d}|y|^2\Big|\frac{1}{N}\frac{1}{\lfloor t\rfloor_{\tau}+1}\sum_{j=1}^{N}\sum_{i=0}^{\lfloor t\rfloor_{\tau}}\Big(\phi_{r}(y-\tilde{X}^{j,\nu}_{i\tau})-\mathbb{E}\phi_{r}(y-\tilde{X}^{j,\nu}_{i\tau})\Big)\Big| \mathrm{d}y \leq \tilde{J}_1+\tilde{J}_2,
\end{align}
where 
\begin{align*}
&\tilde{J}_1=C \int_{|y|>R}|y|^2 \mathbb{E}\Big|\frac{1}{N}\frac{1}{\lfloor t\rfloor_{\tau}+1}\sum_{j=1}^{N}\sum_{i=0}^{\lfloor t\rfloor_{\tau}}\left(\phi_{r}(y-\tilde{X}^{j,\nu}_{i\tau})-\mathbb{E}\phi_{r}(y-\tilde{X}^{j,\nu}_{i\tau})\right)\Big|  \mathrm{d} y,
\\& \tilde{J}_2=C \int_{|y| \leq R}|y|^2 \mathbb{E}\Big|\frac{1}{N}\frac{1}{\lfloor t\rfloor_{\tau}+1}\sum_{j=1}^{N}\sum_{i=0}^{\lfloor t\rfloor_{\tau}}\left(\phi_{r}(y-\tilde{X}^{j,\nu}_{i\tau})-\mathbb{E}\phi_{r}(y-\tilde{X}^{j,\nu}_{i\tau})\right)\Big| \mathrm{d} y.
\end{align*}
Owing to the non-negativity of density functions and the identical distribution property of $\tilde{X}^{j,\nu}$, we obtain that
\begin{align*}
\tilde{J}_1 & \leq  \frac{1}{N} \sum_{j=1}^{N}\left\{\frac{C}{\lfloor t\rfloor_{\tau}+1}\sum_{i=0}^{\lfloor t\rfloor_{\tau}} \int_{|y|>R}|y|^2\mathbb{E}\left[\phi_r(y-\tilde{X}^{j,\nu}_{i\tau})+\mathbb{E}\phi_r(y-\tilde{X}^{j,\nu}_{i\tau})\right]\mathrm{d} y\right\}
\\&=\frac{C}{\lfloor t\rfloor_{\tau}+1}\sum_{i=0}^{\lfloor t\rfloor_{\tau}} \int_{|y|>R}|y|^2\mathbb{E}\left[\phi_r(y-\tilde{X}^{j,\nu}_{i\tau})+\mathbb{E}\phi_r(y-\tilde{X}^{j,\nu}_{i\tau})\right]\mathrm{d} y.
\end{align*}
It follows from  \eqref{eqJ} and \eqref{e2.8} that
\begin{align}\label{e5.6}
\tilde{J}_1 \leq C r^2+C_{\nu}(R-r)^{-\rho}.
\end{align}
For $\tilde{J}_2$, since $\tilde{X}^{j,\nu}, 1\leq j\leq N$ are i.i.d, using the H\"older inequality  we compute that for any $1\leq j\leq N$,
\begin{align}
\tilde{J}_2&=C \int_{|y| \leq R}|y|^2 \mathbb{E}\Big|\frac{1}{N}\sum_{j=1}^{N}\Big(\frac{1}{\lfloor t\rfloor_{\tau}+1}\sum_{i=0}^{\lfloor t\rfloor_{\tau}}h(i\tau, \tilde{X}^{j,\nu}_{i\tau},y)\Big)\Big|\mathrm{d} y\nn\
\\ &\leq C \int_{|y| \leq R}|y|^2 \Big[\mathbb{E}\Big|\frac{1}{N}\sum_{j=1}^{N}\Big(\frac{1}{\lfloor t\rfloor_{\tau}+1}\sum_{i=0}^{\lfloor t\rfloor_{\tau}}h(i\tau, \tilde{X}^{j,\nu}_{i\tau}, y)\Big)\Big|^2\Big]^{\frac{1}{2}} \mathrm{d} y\nn\\
 &\leq \frac{CN^{-\frac{1}{2}}}{(\lfloor t\rfloor_{\tau}+1)}\int_{|y| \leq R}|y|^2 \mathbb{E}\Big[\sum_{i=0}^{\lfloor t\rfloor_{\tau}}\sum_{l=0}^{\lfloor t\rfloor_{\tau}}h(i\tau, \tilde{X}^{j,\nu}_{i\tau}, y) h(l\tau, \tilde{X}^{j,\nu}_{l\tau}, y) \Big]^{\frac{1}{2}} \mathrm{d} y. \nn\
 \end{align}
 Then utilizing \eqref{e2.9} and \eqref{e2.10} yields that
 \begin{align}\label{e5.7}
\tilde{J}_2\leq \Big[ \frac{C_{\nu}N^{-\frac{1}{2}}R^{d+2}r^{-d}}{\sqrt{t}}+\frac{C_{\nu}N^{-\frac{1}{2}}R^{d+2}r^{-d}}{\sqrt{\lfloor t\rfloor_{\tau}+1}}\Big].
\end{align}
Combining \eqref{e5.5}-\eqref{e5.7} leads to that
\begin{align}
&\mathbb{E}\Big[\mathcal{W}^2_{2}\Big(\frac{1}{N}\sum_{j=1}^{N}\mathcal{E}^{\tilde{X}^{j,\nu}}_{\tau,t,r},\frac{1}{N}\sum_{j=1}^{N}\mathcal{L}^{\tilde{X}^{j,\nu}}_{\tau,t,r}\Big)\Big]\nn
\\ \leq& C r^2+C_{\nu}(R-r)^{-\rho}+ \frac{C_{\nu}N^{-\frac{1}{2}}R^{d+2}r^{-d}}{\sqrt{t}}+\frac{C_{\nu}N^{-\frac{1}{2}}R^{d+2}r^{-d}}{\sqrt{\lfloor t\rfloor_{\tau}+1}}.
\end{align}
This, together with \eqref{e2.5}-\eqref{J3.22}, gives that
\begin{align}
\mathbb{E}\Big[\mathcal{W}^2_2\Big(\frac{1}{N}\sum_{j=1}^{N}\mathcal{E}^{\tilde{X}^{j,\nu}}_{\tau,t}, \mu^*\Big)\Big] &\leq Cr^2+C_{\nu}\Big(\tau+\frac{1}{\kappa_1}\Big)\frac{1}{t}+C_{\nu}(R-r)^{-\rho}\nn\
\\&~~~+ \frac{C_{\nu}N^{-\frac{1}{2}}R^{d+2}r^{-d}}{\sqrt{t}}+\frac{C_{\nu}N^{-\frac{1}{2}}\tau R^{d+2} r^{-d}}{\sqrt{t}}.
\end{align}
Taking
$$
r=t^{-\frac{\eta}{2}}N^{-\frac{1}{2(d+2)}}  \quad \text { and } \quad R=t^{\frac{1}{(d+2)(\rho+2)} },
$$
implies that for each $\nu \in \mathscr{P}_{2+\rho}(\RR^d)$,
$$
\mathbb{E}\Big[\mathcal{W}^2_2\Big(\frac{1}{N}\sum_{j=1}^{N}\mathcal{E}^{\tilde{X}^{j,\nu}}_{\tau,t}, \mu^*\Big)\Big] \leq C_{\nu}\left(1+\tau\right) t^{-\eta}N^{-\frac{1}{d+2}}.
$$
 The proof is complete.
\end{proof}
\begin{thm}\label{T3.2}
Under Assumptions \ref{as2}, for any initial distribution $\nu\in \mathscr{P}_{2+\rho}(\RR^{d})$ and  $\delta\in (0,1-\bar{\kappa}_2/\bar{\kappa}_1)$, there exists a positive constant $C_{\nu,\delta}$ such that
\begin{align*}
\mathbb{E}\Big[\mathcal{W}^2_{2}\Big(\frac{1}{N}\sum_{j=1}^{N}\mathcal{E}^{Y^{j,N,\nu}}_{\tau,t},\mu^{*}\Big)\Big]\leq C_{\nu,\delta} \left(1+\tau\right) t^{-\eta\wedge \delta}N^{-\frac{1}{d+2}}.
\end{align*}
\end{thm}
\begin{proof}
Applying the elementary inequality and Lemma \ref{L3.2}, it is sufficient to estimate 
\begin{align*} \mathcal{A}:&=\mathbb{E}\Big[\mathcal{W}_{2}^{2}\Big(\frac{1}{N}\sum_{j=1}^{N}\mathcal{E}^{Y^{j,N,\nu}}_{\tau,t},\frac{1}{N}\sum_{j=1}^{N}\mathcal{E}^{\tilde{X}^{j,\nu}}_{\tau,t}\Big)\Big]\nn\
\\&\leq \frac{1}{N}\sum_{j=1}^{N}\mathbb{E}\Big[\mathcal{W}_2^{2}\big(\mathcal{E}^{Y^{j,N,\nu}}_{\tau,t}, \mathcal{E}^{\tilde{X}^{j,\nu}}_{\tau,t}\big)\Big]\nn\
\\&\leq \frac{1}{N}\sum_{j=1}^{N}\int_{0}^{t}\mathbb{E}|Y^{j,N,\nu}_{ts}-\tilde{X}^{j,\nu}_{ts}|^{2}\pi_{\tau,t}(\mathrm{d}s).
\end{align*}
for the desired result. Since $Y^{j,N,\nu}_{t}-\tilde{X}^{j,\nu}_{t}, 1\leq j\leq N$ are identically distributed for any $t\geq0$,
\begin{align}\label{eq3.25}
\mathcal{A}\leq \int_{0}^{t}\mathbb{E}|Y^{j,N,\nu}_{ts}-\tilde{X}^{j,\nu}_{ts}|^{2}\pi_{\tau,t}(\mathrm{d}s) ,~~~\forall 1\leq j\leq N.
\end{align}
Utilizing the It\^o formula  yields that for any $t\geq0$,
\begin{align*}
\mathbb{E}\big|Y^{j,N,\nu}_t-\tilde{X}^{j,\nu}_t\big|^2&=\int_{0}^{t}\mathbb{E}\Big[2(Y^{j,N,\nu}_{s}-\tilde{X}^{j,\nu}_{s})^{T}\Big(f\big(Y^{j,N,\nu}_{s},\frac{1}{N}\sum_{j=1}^{N}\mathcal{E}^{Y^{j,N,\nu}}_{\tau,s}\big)-f\big(\tilde{X}^{j,\nu}_{s},\mu^*)\Big)\nn\
\\&~~~+\Big|g\big(Y^{j,N,\nu}_{s},\frac{1}{N}\sum_{j=1}^{N}\mathcal{E}^{Y^{j,N,\nu}}_{\tau,s}\big)-g\big(\tilde{X}^{j,\nu}_{s},\mu^{*}\big)\Big|^2\Big]\mathrm{d}s.
\end{align*}
Then by Assumption \ref{as2} we obtain that
\begin{align}\label{eq3.12}
\frac{\mathrm{d}}{\mathrm{d}t} \mathbb{E}\big|Y^{j,N,\nu}_t-\tilde{X}^{j,\nu}_t\big|^2&= \mathbb{E}\Big[2(Y^{j,N,\nu}_{t}-\tilde{X}^{j,\nu}_{t})^{T}\Big(f\big(Y^{j,N,\nu}_{t},\frac{1}{N}\sum_{j=1}^{N}\mathcal{E}^{Y^{j,N,\nu}}_{\tau,t}\big)-f(\tilde{X}^{j,\nu}_{t},\mu^*)\Big)\nn\
\\&~~~+\Big|g\big(Y^{j,N,\nu}_{t},\frac{1}{N}\sum_{j=1}^{N}\mathcal{E}^{Y^{j,N,\nu}}_{\tau,t}\big)-g(\tilde{X}^{j,\nu}_{t},\mu^{*})\Big|^2\Big]\nn\
\\ &\leq-\bar{\kappa}_1 \mathbb{E}\big|Y^{j,N,\nu}_t-\tilde{X}^{j,\nu}_t\big|^2+\bar{\kappa}_2 \mathbb{E}\Big[\mathcal{W}^2_2\Big(\frac{1}{N}\sum_{j=1}^{N}\mathcal{E}^{Y^{j,N,\nu}}_{\tau,t}, \mu^*\Big)\Big].
\end{align}
 For any for any $\delta\in (0,1-\bar{\kappa}_2/\bar{\kappa}_1)$, choose a constant $\tilde{\kappa}_2\in \mathcal{A}_{\delta}$. Thus,  
 $\bar{\kappa}_2<\tilde{\kappa}_2<\bar{\kappa}_1$ and 
 \begin{align*}
\limsup_{t\rightarrow\infty}\int_{0}^{1}s^{-\delta}\pi_{\tau,t}(\mathrm{d}s)< \frac{\bar{\kappa}_1}{\tilde{\kappa}_2}.
\end{align*}
Then similar to deriving \eqref{e3.14}, we also deduce that
\begin{align*}
&\frac{\mathrm{d}}{\mathrm{d}t} \mathbb{E}\big|Y^{j,N,\nu}_t-\tilde{X}^{j,\nu}_t\big|^2 \nn\
\\
\leq&  -\bar{\kappa}_1 \mathbb{E}\big|Y^{j,N,\nu}_t-\tilde{X}^{j,\nu}_t\big|^2+\tilde{\kappa}_{2}\mathbb{E}\Big[\mathcal{W}^2_{2}\Big(\frac{1}{N}\sum_{j=1}^{N}\mathcal{E}^{Y^{j,N,\nu}}_{\tau,t},\frac{1}{N}\sum_{j=1}^{N}\mathcal{E}^{\tilde{X}^{j,\nu}}_{\tau,t}\Big)\Big]\nn\
\\&~~~+C\mathbb{E}\Big[\mathcal{W}^2_2\Big(\frac{1}{N}\sum_{j=1}^{N}\mathcal{E}^{\tilde{X}^{j,\nu}}_{\tau,t}, \mu^*\Big)\Big]\nn\
\\ \leq &-\bar{\kappa}_1 \mathbb{E}|Y^{j,N,\nu}_{t}-\tilde{X}^{j,\nu}_{t}|^2+\tilde{\kappa}_2\int_{0}^{1}\mathbb{E}|Y^{j,N,\nu}_{ts}-\tilde{X}^{j,\nu}_{ts}|^2\pi_{\tau,t}(\mathrm{d}s)+ C_{\nu} \left(1+\tau\right) t^{-\eta}N^{-\frac{1}{d+2}}
\end{align*}
for any $1\leq j\leq N$. 
This, along with Lemma \ref{L5}, gives that
\begin{align}
\mathbb{E}|Y^{j,N,\nu}_t-\tilde{X}^{j,\nu}_t|^2&\leq C_{\nu,\delta} \left(1+\tau\right) t^{-\eta \wedge \delta}N^{-\frac{1}{d+2}}.
\end{align}
Then we obtain that
\begin{align*}
\mathcal{A}&\leq \int_{0}^{t}\mathbb{E}|Y^{j,N,\nu}_{ts}-\tilde{X}^{j,\nu}_{ts}|^2\pi_{\tau,t}(\mathrm{d}s)\leq C_{\nu,\delta} \left(1+\tau\right) t^{-\eta\wedge\delta}N^{-\frac{1}{d+2}}\int_{0}^{1}s^{-\eta\wedge\delta}\pi_{\tau,t}(\mathrm{d}s)\nn\
\\&\leq C_{\nu,\delta} \left(1+\tau\right) t^{-\eta\wedge\delta}N^{-\frac{1}{d+2}}.
\end{align*}
This, together with Lemma \ref{L3.2}, implies the desired result. The proof is complete.
\end{proof}
\subsection{The  EM scheme in the averaged weighed empirical approximation}
 This subsection will further design the EM scheme for the multi-particle system  and  study the convergence rate of the EM schemes in the averaged weighed  empirical approximation of  the invariant probability measure $\mu^*$ of MV-SDE \eqref{eq1} in  $\mathcal{W}_2$-Wasserstein distance.

 For any fixed $\tau>0$, let $M\in \mathbb{N}_{+}$ be  sufficiently large such that the
step size $\Delta=\tau/M\in (0,1]$. Let $\{X^{j}_{0}\}_{j=1}^{N}$ be independent copies of $X_0$ on probability space $(\Omega, \mathscr{F}, \mathbb{P})$. Then for any $k\in \mathbb{N}$, 
\begin{equation}
\begin{cases}
Z^{j,N}_0=X^{j}_0, ~~1\leq j\leq N,\\
\frac{1}{N}\sum_{j=1}^{N}\mathcal{E}^{Z^{j,N}}_{\tau,k\tau}=\frac{1}{N}\sum_{j=1}^{N}\big(\frac{1}{k+1}\sum_{i=0}^{k}\boldsymbol{\delta}_{Z^{j,N}_{i\tau}}\big),\\
Z^{j,N}_{k\tau+(m+1)\Delta}=Z^{j,N}_{k\tau+m\Delta}+f\big(Z^{j,N}_{k\tau+m\Delta},\frac{1}{N}\sum_{j=1}^{N}\mathcal{E}^{Z^{j,N}}_{\tau,k\tau}\big)\Delta
\nn\
\\~~~~~~~~~~~~~~~~~~~~~+g\big(Z^{j,N}_{k\tau+m\Delta}, \frac{1}{N}\sum_{j=1}^{N}\mathcal{E}^{Z^{j,N}}_{\tau,k\tau}\big)\sqrt{\Delta} \xi^{j}_{k,m},~~1\leq j\leq N,~~m=0,1,\cdots,M-1,
\end{cases}
\end{equation}
 where  $\{\xi^{j}_{k,m}\}$ is the sequences of i.i.d. Gaussian random
variables with mean 0 and variance 1 on probability space $(\Omega, \mathscr{F}, \mathbb{P})$. Define two versions of numerical solutions  by
\begin{align}\label{e18}
Z^{j,N}_{t}=Z^{j,N}_{\lfloor t\rfloor_{\Delta}\Delta}, ~~1\leq j\leq N,
\end{align}
and
\begin{align}\label{e19}
\bar{Z}^{j,N}_t&=X^{j}_{0}+\int_{0}^{t}f\Big(Z^{j,N}_s, \frac{1}{N}\sum_{j=1}^{N}\mathcal{E}^{Z^{j,N}}_{\tau,s}\Big)\mathrm{d}s+\int_{0}^{t}g\Big(Z^{j,N}_s, \frac{1}{N}\sum_{j=1}^{N}\mathcal{E}^{Z^{j,N}}_{\tau,s}\Big)\mathrm{d}B^{j}_{s},~~1\leq j\leq N,
\end{align}
for any $t\geq0$. One observes that for any $t\in [k\tau, (k+1)\tau),  k\in \mathbb{N}$, 
 $$\frac{1}{N}\sum_{j=1}^{N}\mathcal{E}^{Z^{j,N}}_{\tau,t}=\frac{1}{N}\sum_{j=1}^{N}\mathcal{E}^{Z^{j,N}}_{\tau,k\tau},$$
and thus  $\bar{Z}^{j,N}_{k\Delta}=Z^{j,N}_{k\Delta}$ for all $k\in \mathbb{N}$ and $1\leq j\leq N$.  Thus we have
\begin{align}\label{e5.12}
 \bar{Z}^{j,N}_{t}=X^{j}_0+\int_{0}^{t}f\Big(Z^{j,N}_s, \frac{1}{N}\sum_{j=1}^{N}\mathcal{E}^{\bar{Z}^{j,N}}_{\tau,s}\Big)\mathrm{d}s+\int_{0}^{t}g\Big(Z^{j,N}_s,\frac{1}{N}\sum_{j=1}^{N}\mathcal{E}^{\bar{Z}^{j,N}}_{\tau,s}\Big)\mathrm{d}B^{j}_s.
 \end{align}
\begin{lem}\label{L5.2}
Let Assumptions \ref{as2} and \ref{as3} hold. Then there exists a constant $\Delta^{**}\in (0,1]$ such that for any $\Delta\in (0,\Delta^{**}]$ and initial distribution $\nu\in \mathscr{P}_{2+\rho}(\RR^{d})$, the numerical solution $\bar{Z}$ defined by \eqref{e18} and \eqref{e19}  satisfies that
\begin{align*}
\sup_{t\geq0}\sup_{1\leq j\leq N}\mathbb{E}|\bar{Z}^{j,N,\nu}_{t}|^{2}\leq C_{\nu}.
\end{align*}
\end{lem}
\begin{proof}
Using the It\^o formula for \eqref{e5.12} yields that
\begin{align}\label{e5.13}
\frac{\mathrm{d}\mathbb{E}|\bar{Z}^{j,N,\nu}_{t}|^{2}}{\mathrm{d}t}=\mathbb{E}\Big[\big(2(\bar{Z}^{j,N,\nu}_{t})^{T}f\Big(Z^{j,N,\nu}_{t}, \frac{1}{N}\sum_{j=1}^{N}\mathcal{E}^{\bar{Z}^{j,N,\nu}}_{\tau,t}\Big)+\Big|g\Big(Z^{j,N,\nu}_{t},\frac{1}{N}\sum_{j=1}^{N}\mathcal{E}^{\bar{Z}^{j,N,\nu}}_{\tau,t}\Big)\Big|^2\Big].
\end{align}
Utilizing the Young inequality and then Assumption \ref{as3} shows that for any $\epsilon\in (0,\kappa_1-\kappa_2)$,
\begin{align}
&2(\bar{Z}^{j,N,\nu}_{t})^{T}f\Big(Z^{j,N,\nu}_{t}, \frac{1}{N}\sum_{j=1}^{N}\mathcal{E}^{\bar{Z}^{j,N,\nu}}_{\tau,t}\Big)+\Big|g\Big(Z^{j,N,\nu}_{t},\frac{1}{N}\sum_{j=1}^{N}\mathcal{E}^{\bar{Z}^{j,N,\nu}}_{\tau,t}\Big)\Big|^2\nn\
\\ \leq&2(\bar{Z}^{j,N,\nu}_{t})^{T}f\Big(\bar{Z}^{j,N,\nu}_{t}, \frac{1}{N}\sum_{j=1}^{N}\mathcal{E}^{\bar{Z}^{j,N,\nu}}_{\tau,t}\Big)+(1+\rho)\Big|g\Big(\bar{Z}^{j,N,\nu}_{t},\frac{1}{N}\sum_{j=1}^{N}\mathcal{E}^{\bar{Z}^{j,N,\nu}}_{\tau,t}\Big)\Big|^2+\epsilon|\bar{Z}^{j,N,\nu}_{t}|^2\nn\
\\& ~~~+C\Big[\Big|f\Big(Z^{j,N,\nu}_{t},\frac{1}{N}\sum_{j=1}^{N}\mathcal{E}^{\bar{Z}^{j,N,\nu}}_{\tau,t}\Big)-f\Big(\bar{Z}^{j,N,\nu}_t, \frac{1}{N}\sum_{j=1}^{N}\mathcal{E}^{\bar{Z}^{j,N,\nu}}_{\tau,t}\Big)\Big|^2\nn\
\\&~~~+\Big|g\Big(Z^{j,N,\nu}_{t},\frac{1}{N}\sum_{j=1}^{N}\mathcal{E}^{\bar{Z}^{j,N,\nu}}_{\tau,t}\Big)-g\Big(\bar{Z}^{j,N,\nu}_t, \frac{1}{N}\sum_{j=1}^{N}\mathcal{E}^{\bar{Z}^{j,N,\nu}}_{\tau,t}\Big)\Big|^{2}\Big]\nn\
\\ \leq & -(\kappa_1-\epsilon)|\bar{Z}^{j,N,\nu}_{t}|^{2}+(\kappa_2+\epsilon)\frac{1}{N}\sum_{j=1}^{N}\mathcal{E}^{\bar{Z}^{j,N,\nu}}_{\tau,t}(|\cdot|^2)+C|\bar{Z}^{j,N,\nu}_t-Z^{j,N,\nu}_t|^{2}+C.\nn\
\end{align}
Taking expectation on both sides of the above inequality and using the identical distribution property yield that
\begin{align}
&2\mathbb{E}\Big[(\bar{Z}^{j,N,\nu}_{t})^{T}f\Big(Z^{j,N,\nu}_{t}, \frac{1}{N}\sum_{j=1}^{N}\mathcal{E}^{\bar{Z}^{j,N,\nu}}_{\tau,t}\Big)+\Big|g\Big(Z^{j,N,\nu}_{t},\frac{1}{N}\sum_{j=1}^{N}\mathcal{E}^{\bar{Z}^{j,N,\nu}}_{\tau,t}\Big)\Big|^2\Big]\nn\
\\ \leq & -(\kappa_1-\epsilon)\mathbb{E}|\bar{Z}^{j,N,\nu}_{t}|^{2}+(\kappa_2+\epsilon)\mathbb{E}\big[\mathcal{E}^{\bar{Z}^{j,N,\nu}}_{\tau,t}(|\cdot|^2)\big]+C\mathbb{E}|\bar{Z}^{j,N,\nu}_t-Z^{j,N,\nu}_t|^{2}+C,~~~1\leq j\leq N.\nn\
\end{align}
Inserting the above inequality into \eqref{e5.13} and using \eqref{e2.1} we arrive at that
\begin{align}\label{e4.8}
\frac{\mathrm{d}\mathbb{E}|\bar{Z}^{j,N,\nu}_{t}|^{2}}{\mathrm{d}t}&=-(\kappa_1-\epsilon)\mathbb{E}|\bar{Z}^{j,N,\nu}_{t}|^2+(\kappa_2+\epsilon)\int_{0}^{1}\mathbb{E}|\bar{Z}^{j,N,\nu}_{ts}|^2\pi_{\tau,t}(\mathrm{d}s)\nn\
\\&~~~+C\mathbb{E}|\bar{Z}^{j,N,\nu}_{t}-Z^{j,N,\nu}_{t}|^2\mathrm{d}s+C.
\end{align}
Using the identical distribution property again, we could deduce that there exists a constant $\Delta^{**}\in (0,1]$ such that for any $\Delta\in (0,\Delta^{**}]$,
\begin{align}\label{eq5.15}
\mathbb{E}|\bar{Z}^{j,N,\nu}_{t}-Z^{j,N,\nu}_{t}|^2
\leq  C\Delta+ C\Delta\mathbb{E}|\bar{Z}^{j,N,\nu}_t|^2+C\Delta\int_{0}^{1} \mathbb{E}|\bar{Z}^{j,N,\nu}_{ts}|^{2}\pi_{\tau,t}(\mathrm{d}s).
\end{align}
The remaining proof is similar to that of Lemma \ref{L7} and thus we omit it to avoid repetition. The proof is complete.
\end{proof}
Applying Lemma \ref{L5.2} and  proceeding with a similar argument to that of Corollary \ref{Cor4.1}, we directly obtain the error of $\bar{Z}^{j,N}_{t}$ and $Z^{j,N}_{t}$ for any $1\leq j\leq N$ and $t\geq0$.
\begin{coro}\label{cor2}
If Assumptions \ref{as2} and \ref{as3} hold, then for any initial distribution $\nu\in \mathcal{P}_{2+\rho}(\RR^{d})$ and $\Delta\in (0,\Delta^{**}]$,  there exists a constant $C_{\nu}$ such that 
\begin{align*}
\sup_{t\geq0}\sup_{1\leq j\leq N}\mathbb{E}|\bar{Z}^{j,N,\nu}_{t}-Z^{j,N,\nu}_{t}|^{2}\leq C_{\nu}\Delta.
\end{align*}
\end{coro}
\begin{thm}\label{L4.1}
Let  Assumptions \ref{as2} and \ref{as3} hold. Then for any initial distribution $\nu\in \mathscr{P}_{2+\rho}(\RR^{d})$ and $\Delta\in (0,\Delta^{**}]$, there exists a constant $C_{\nu}$ such that
\begin{align*}
\sup_{t\geq0}\mathbb{E}\big|Y^{j,N,\nu}_t-Z^{j,N,\nu}_t\big|^2\leq C_{\nu}\Delta.
\end{align*}
\end{thm}
\begin{proof}
According to \eqref{eq5.1} and  \eqref{e5.12},
employing the It\^o formula and \eqref{2.10} leads to that
\begin{align}\label{eq5.18}
&\frac{\mathrm{d}\mathbb{E}|Y^{j,N,\nu}_{t}-\bar{Z}^{j,N,\nu}_{t}|^{2}}{\mathrm{d}t}\nn\
\\=&\mathbb{E}\Big[2\big(Y^{j,N,\nu}_{t}-\bar{Z}^{j,N,\nu}_{t}\big)^{T}\Big(f\Big(Y^{j,N,\nu}_{t},\frac{1}{N}\sum_{j=1}^{N}\mathcal{E}^{Y^{j,N,\nu}}_{\tau,t}\Big)-f\Big(Z^{j,N,\nu}_{t},\frac{1}{N}\sum_{j=1}^{N}\mathcal{E}^{\bar{Z}^{j,N,\nu}}_{\tau,t}\Big)\Big)\nn\
\\&~~~+\Big|g\Big(Y^{j,N,\nu}_t,\frac{1}{N}\sum_{j=1}^{N}\mathcal{E}^{Y^{j,N,\nu}}_{\tau,t}\Big)-g\Big(Z^{j,N,\nu}_t,\frac{1}{N}\sum_{j=1}^{N}\mathcal{E}^{\bar{Z}^{j,N,\nu}}_{\tau,t}\Big)\Big|^2\Big]\leq \mathcal{A}_1+\mathcal{A}_2+\mathcal{A}_3,
\end{align}
where
\begin{align*}
&\mathcal{A}_1=\mathbb{E}\Big[2\big(Y^{j,N,\nu}_{t}-\bar{Z}^{j,N,\nu}_{t}\big)^{T}\Big(f\Big(Y^{j,N,\nu}_{t},\frac{1}{N}\sum_{j=1}^{N}\mathcal{E}^{Y^{j,N,\nu}}_{\tau,t}\Big)-f\Big(\bar{Z}^{j,N,\nu}_{t},\frac{1}{N}\sum_{j=1}^{N}\mathcal{E}^{\bar{Z}^{j,N,\nu}}_{\tau,t}\Big)\Big)\nn\
\\&~~~~~+\Big|g\Big(Y^{j,N,\nu}_t,\mathcal{E}^{Y^{j,N,\nu}}_{\tau,t}\Big)-g\Big(\bar{Z}^{j,N,\nu}_t,\frac{1}{N}\sum_{j=1}^{N}\mathcal{E}^{\bar{Z}^{j,N,\nu}}_{\tau,t}\Big)\Big|^2\Big],\nn\
\end{align*}
\begin{align*}
&\mathcal{A}_2=2\mathbb{E}\Big[\big(Y^{j,N,\nu}_{t}-\bar{Z}^{j,N,\nu}_{t}\big)^{T}\Big(f\Big(\bar{Z}^{j,N,\nu}_{t},\frac{1}{N}\sum_{j=1}^{N}\mathcal{E}^{\bar{Z}^{j,N,\nu}}_{\tau,t}\Big)-f\Big(Z^{j,N,\nu}_{t},\frac{1}{N}\sum_{j=1}^{N}\mathcal{E}^{\bar{Z}^{j,N,\nu}}_{\tau,t}\Big)\Big)\Big]\nn\
\\&~~~+\mathbb{E}\Big|g\Big(\bar{Z}^{j,N,\nu}_{t},\frac{1}{N}\sum_{j=1}^{N}\mathcal{E}^{\bar{Z}^{j,N,\nu}}_{\tau,t}\Big)-g\Big(Z^{j,N,\nu}_{t},\frac{1}{N}\sum_{j=1}^{N}\mathcal{E}^{\bar{Z}^{j,N,\nu}}_{\tau,t}\Big)\Big|^2,
\end{align*}
and
\begin{align*}
&\mathcal{A}_3=2\mathbb{E}\Big(\Big|g\Big(Y^{j,N,\nu}_t,\frac{1}{N}\sum_{j=1}^{N}\mathcal{E}^{Y^{j,N,\nu}}_{\tau,t}\Big)-g\Big(\bar{Z}^{j,N,\nu}_t,\frac{1}{N}\sum_{j=1}^{N}\mathcal{E}^{\bar{Z}^{j,N,\nu}}_{\tau,t}\Big)\Big|\nn\
\\&~~~~~~~~\times\Big|g\Big(\bar{Z}^{j,N,\nu}_t,\frac{1}{N}\sum_{j=1}^{N}\mathcal{E}^{\bar{Z}^{j,N,\nu}}_{\tau,t}\Big)-g\Big(Z^{j,N,\nu}_t,\frac{1}{N}\sum_{j=1}^{N}\mathcal{E}^{\bar{Z}^{j,N,\nu}}_{\tau,t}\Big)\Big|\Big).\nn\
\end{align*}
As we did in the proof of Theorem \ref{LJP4.1},  using the identical distribution property and Corollary \ref{cor2}, we also deduce that for any $1\leq j\leq N$,
\begin{align*}
&\frac{\mathrm{d}\mathbb{E}\big|Y^{j,N,\nu}_{t}-\bar{Z}^{j,N,\nu}_{t}\big|^{2}}{\mathrm{d}t}\nn\
\\ \leq & -\left(\bar{\kappa}_1-\varepsilon\right)\mathbb{E}\big|Y^{j,N,\nu}_t-\bar{Z}^{j,N,\nu}_{t}\big|^{2}+\left(\bar{\kappa}_2+\varepsilon\right)\int_{0}^{1}\mathbb{E}\big|Y^{j,N,\nu}_{ts}-\bar{Z}^{j,N,\nu}_{ts}\big|^{2}\pi_{\tau,t}(\mathrm{d}s)+C_{\nu}\Delta.
\end{align*}
Then remaining proof is same as that of Theorem \ref{LJP4.1} and thus we omit it. The proof is complete.
\end{proof}
Combining Theorems \ref{T3.2} and \ref{L4.1}, we obtain the desired convergence rate between the averaged weighted empirical measure of multi-particle system~and invariant probability measure $\mu^*$ directly.
\begin{thm}\label{L5.1}
Let  Assumptions \ref{as2} and \ref{as3} hold. Then for any initial distribution $\nu\in \mathscr{P}_{2+\rho}(\RR^{d})$,  $\Delta\in (0,\Delta^{**}]$ and $\delta\in (0,1-\bar{\kappa}_2/\bar{\kappa}_1)$, there exists a constant $C_{\nu,\delta}$ such that 
\begin{align*}
\mathbb{E}\Big[\frac{1}{N}\sum_{j=1}^{N}\mathcal{W}_{2}^{2}\big(\mathcal{E}^{Z^{j,N,\nu}}_{\tau,t},\mu^*\big)\Big]\leq  C_{\nu,\delta}\left[ \left(1+\tau\right) t^{-\eta\wedge\delta}N^{-\frac{1}{d+2}}+\Delta\right].
\end{align*}
\end{thm}
\section{Computational cost}
This section concerns  the computational cost of the EM scheme in tow types of empirical approximation: the  weighted empirical approximation (WEA) and the averaged weighted empirical approximation (AWEA). For convenience, we  use  $A:=A(\eta)$  to denote an algorithm that outputs the approximation for the invariant probability measure $\mu^*$ of MV-SDE \eqref{eq1}, where $\eta$ denotes the set of  parameters required for implementing this algorithm.  Next, we denote the  error in $\mathcal{W}_2$-Wasserstein distance as 
\begin{align*}
Error(A(\eta))=\Big(\mathbb{E}\mathcal{W}^2_{2}(A(\eta),\mu^*)\Big)^{\frac{1}{2}}.
\end{align*}
On the other hand, denote by  $Cost(A)$ the computational cost of the algorithm $A$. To facilitate comparison of computational cost,  for a fixed error tolerance of $\epsilon>0$,  we solve the following optimization problem 
\begin{equation}
\begin{cases}
Error(A(\eta))<\epsilon,\\
\min\limits_{\eta}Cost(A(\eta)),
\end{cases}
\end{equation}
then we compute the minimum cost of the EM scheme  in two types of empirical approximations.
First, we consider the computational cost of EM scheme in the WEA. For any $\epsilon>0$, in view of Theorem \ref{L4.1}, choosing that
 $$\tau\approx 1,~\Delta\approx \epsilon^2,~t\approx \epsilon^{-\frac{2}{\eta\wedge\delta}},~~~\hbox{where}~\delta\in (0,1-\bar{\kappa}_2/\bar{\kappa}_1),$$
 ensures $$Error(A^{WEA}(t,\tau,\Delta))\lesssim\epsilon.$$ 
 As the cost of simulating  self-interacting   at every step  in interval $[k\tau, (k+1)\tau)$ is $k+1$,  the computational cost of EA is
$$Cost^{WEA}(t, \tau, \Delta)\approx \frac{t}{\Delta}+\frac{\tau}{\Delta}\big(2+\cdots+t/\tau+1)\approx \frac{t}{\Delta}\frac{t}{\tau}\approx \epsilon^{-(\frac{4}{\eta\wedge\delta}+2)},~~~\delta\in (0,1-\bar{\kappa}_2/\bar{\kappa}_1).$$
\begin{remark}
For the given error tolerance $\epsilon>0$,  if we choose the parameters of WEA as
$$\tau\approx\Delta, ~\Delta\approx \epsilon^2,~t\approx \epsilon^{-\frac{2}{\eta\wedge\delta}}$$
to ensure $Error^{WEA}(t,\tau,\Delta)\lesssim\epsilon$. One  computes that $$Cost^{WEA}( t, \tau, \Delta)\approx \frac{t^{2}}{\Delta^{2}}\approx \epsilon^{-\frac{4}{\eta\wedge\delta}-4}>\epsilon^{-(\frac{4}{\eta\wedge\delta}+2)}.$$
\end{remark}
Now, we discuss the computational cost of the EM scheme in the AWEA. For any given error tolerance $\epsilon>0$, using Theorem \ref{L5.1} leads to the following choice of the parameters  
$$\tau\approx 1,~\Delta\approx \epsilon^2,~t\approx \epsilon^{-\frac{2r}{\eta\wedge\delta}}, ~N\approx \epsilon^{-2(d+2)(1-r)}, ~~
0\leq r\leq 1,$$
to make $$Error^{AWEA}(A(t,\tau,\Delta,N)) \lesssim\epsilon.$$
As the cost of simulating particle system and self-interacting   at every step  in interval $[k\tau, (k+1)\tau)$ is $(k+1)N$, then  we compute that
\begin{align}
Cost^{AWEA}(N,t,\tau,\Delta)&\approx N\left[\frac{t}{\Delta}+N\frac{\tau}{\Delta}\big(2+\cdots+\frac{t}{\tau}+1\big)\right]\approx N^2 \frac{T}{\Delta}\frac{T}{\tau} \approx \epsilon^{-v(r)},
\end{align}
where $$v(r)=4(d+2)(1-r)+\frac{4r}{\eta\wedge\delta}+2,~~\delta\in (0,1-\bar{\kappa}_2/\bar{\kappa}_1)$$
Then we need to find the best parameter selection of $r$~to minimize  $v(r)$ and thus achieve minimum the  cost. One notes that
\begin{align*}
v'(r)&=-4(d+2)+\frac{4}{\eta\wedge\delta}\geq -4(d+2)+\frac{4}{\eta}\nn\
\\&=-4(d+2)+\frac{4(d+2)(\rho+2)}{\rho}\geq 0
\end{align*}
Thus, 
$$Cost^{AWEA}(t,\tau, \Delta,N)\approx \epsilon^{-v(0)}\approx\epsilon^{-(4d+10)}\leq \epsilon^{-v(1)}=\epsilon^{-\frac{4}{\eta\wedge\delta}+2}\approx Cost^{WEA}(\tau, t,\Delta).$$
The above inequality implies that the computational cost of the AWEA is lower  than  that of the WEA. In addition, the above computations do not take
under consideration the fact that ensemble algorithms can take full advantage of the parallel
computer architecture and therefore will be superior in practice.
\section{Numerical examples}
In this section, we use two numerical examples to illustrate the effectiveness of the EM scheme in two types of empirical approximations to invariant probability  measures of MV-SDEs. 
\begin{expl}
Consider the following self-interacting process
\begin{align*}
\mathrm{d}Y_t=-\Big(5Y_t+\frac{1}{\lfloor t\rfloor_{\tau}+1}\sum_{j=1}^{\lfloor t\rfloor_{\tau}}Y_{j\tau}\Big)\mathrm{d}t+\Big(Y_t-\Big(\frac{1}{\lfloor t\rfloor_{\tau}+1}\sum_{j=1}^{\lfloor t\rfloor_{\tau}}Y^2_{j\tau}\Big)^{\frac{1}{2}}-2\Big)\mathrm{d}B_t,~~~\forall t\geq0.
\end{align*}
Let $\Delta\in (0,1]$ sufficiently small and let $M=\tau/\Delta\in \mathbb{N}_{+}$. Then the EM scheme is defined by 
\begin{equation}\label{n5.1}
\begin{cases}
 Z_0=X_0,~~~k=0,1,\cdots ,\\
 Z_{k\tau +(m+1)\Delta}=Z_{k\tau +m\Delta}-(5Z_{k\tau +m\Delta}+\frac{1}{k+1}\sum_{j=0}^k{Z_{j\tau}})\Delta ,\\
\hspace{2cm} +\big[Z_{k\tau +m\Delta}-(\frac{1}{k+1}\sum_{j=0}^k{|}Z_{j\tau}|^2)^{\frac{1}{2}}\big]\sqrt{\Delta}\xi _{k,m},\quad m=0,1,\cdots, M-1,\\
\end{cases}
\end{equation}
where $\{\xi_{k,m}\}$ is the sequence of i.i.d. Gaussian random variables.  As per Theorem \ref{LJP4.1}, the EM numerical solution  approximates the exact solution $Y$  in the mean square sense with  a uniform 1/2-order convergence rate with respect to time $t$. We conduct
numerical experiments to verify this result  by implementing  \eqref{n5.1} using MATLAB. Since the exact solution $Y$  cannot be computed analytically, we will regard a numerical solution generated by the step size $\Delta=2^{-15}$ as the true solution. Now, we define the root mean square error (RMSE) by 
\begin{align*}
RMSE(t):= \Big(\frac{1}{M}\sum_{i=1}^{M}|Y_{t}-Z_{t}|^{2}\Big)^{\frac{1}{2}}.
\end{align*}
Let $\tau=1$ and $\Delta=2^{-q}, q=5,6,7,8,9,10,11$. Figure\ref{figure1} plots the functions~$\log_2(RMSE)$~of $q=5,6,7,8,9,10,11$ at $t=10, 20, 40$ and $60$, respectively,  while the red dashed is a reference line of slope -1/2.  These results verify that EM approximation of the self-interacting process has a uniform $1/2$-order convergence rate with respect to time $t$.
\begin{figure}[H]
  \centering
\includegraphics[width=12cm,height=8cm]{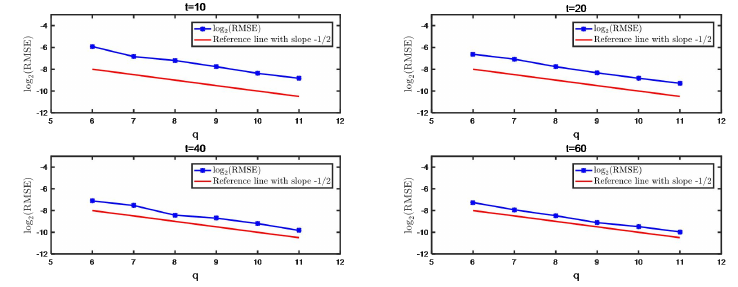}
\captionsetup{font=footnotesize}
  \caption{The functions~$\log_2(RMSE)$~of $q=5,6,7,8,9,10,11$ at $t=10, 20, 40$ and $60$.}\vspace{-1em}
\label{figure1}
\end{figure}
\end{expl}
\begin{expl}
Consider the following MV-SDE
\begin{align}\label{eqq5.1}
\mathrm{d}X_{t}=(-2X_{t}-\mathbb{E}X_{t})\mathrm{d}t+(2-\sqrt{\mathbb{E}|X_{t}|^{2}})\mathrm{d}B(t).
\end{align}
It is straight to verify that MV-SDE \eqref{eqq5.1} satisfies the Assumptions \ref{as2} and \ref{as3}. In addition, from \cite{MR4580925}, the normal distribution $N(0,4/9)$ is the unique invariant probability measure of MV-SDE \eqref{eqq5.1}. For any given $\tau>0$,  the  self-interacting process is
\begin{align}\label{eq6.2}
\mathrm{d}Y_{t}=\Big(-2Y_{t}-\frac{1}{\lfloor t\rfloor_{\tau}+1}\sum_{j=0}^{\lfloor t\rfloor_{\tau}}Y_{j\tau}\Big)\mathrm{d}t+\Big[2-
\Big(\frac{1}{\lfloor t\rfloor_{\tau}+1}\sum_{j=0}^{\lfloor t\rfloor_{\tau}}|Y_{j\tau}|^2\Big)^{\frac{1}{2}}\Big]\mathrm{d}B_{t}.
\end{align}
By Theorem \ref{T3.1},  the weighted empirical measure  of self-interacting process $Y$
 $$\mathcal{E}^{Y}_{\tau,t}(\cdot)=\frac{1}{\lfloor t\rfloor_{\tau}}\sum_{j=1}^{\lfloor t\rfloor_{\tau} }\delta_{Y_{j\tau}}(\cdot)$$ 
 converges to the invariant probability measure $N(0,9/4)$  in $\mathcal{W}_2$-Wasserstein distance. Let $\Delta\in (0,1]$ sufficiently small and let $M=\tau/\Delta\in \mathbb{N}_{+}$. Then the EM scheme for self-interacting process $Y$  is defined by 
\begin{equation}\label{e5.3}
\begin{cases}
&Z_{0}=X_0,~~~k=0,1,\cdots,\\
&Z_{k\tau+(m+1)\Delta}=Z_{k\tau+m\Delta}-\Big(2Z_{k\tau+m\Delta}+\frac{1}{k+1}\sum_{j=0}^{k}Z_{j\tau}\Big)\Delta,\nn\
\\&~~~+\Big[2-\Big(\frac{1}{k+1}\sum_{j=0}^{k}|Z_{j\tau}|^2\Big)^{\frac{1}{2}}\Big]\sqrt{\Delta}\xi^{k}_{m}, ~~~ m=0,1,\cdots, M-1,
\end{cases}
\end{equation}
where $\{\xi^{k}_{m}\}$ is the sequence of i.i.d. Gaussian random variables.
Let $\tau=0.5$ and $\Delta=2^{-8}$. We conduct the Jarque-Bera (J-B) test with a significance level of 0.05  by Matlab to
confirm that the numerical empirical measure $\mathcal{E}^{Z}_{\tau,t}$ follows a normal  distribution for $t= 500, 1000, 2000,4000$. 
\begin{figure}[H]
  \centering
\includegraphics[width=12cm,height=8cm]{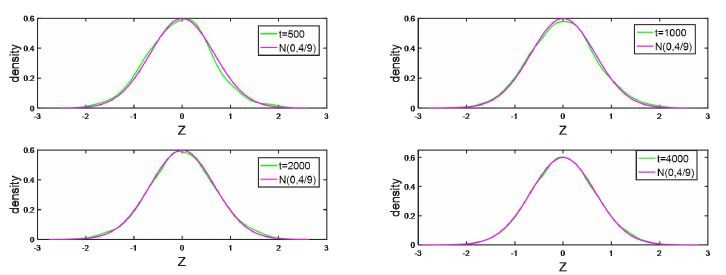}
\captionsetup{font=footnotesize}
  \caption{The empirical density functions of self-interacting process at $t=500, 1000, 2000, 4000$ and the normal distribution N(0, 9/4). }
\label{figure2}
\end{figure}

To more intuitively illustrate the result of Theorem \ref{L4.1},  Figure \ref{figure2} plots numerical empirical density functions of $\mathcal{E}^{Z}_{\tau,t}$ at time  $t=500, 1000, 2000$ and $4000$ the density function of normal distribution $N(0,4/9)$, respectively. As can be seen from Figure \ref{figure2}, as the time $t$ increases, the shape of the empirical density function  becomes increasingly close to the density function of normal distribution $N(0,4/9)$.

In addition,  let $\tau=0.5$, $t=100$ and $\Delta=2^{-8}$. We further compare the averaged weighted empirical measures of the multi-particle systems for $N=1, 50, 100$  and $N=200$. We use the J-B test to examine that the averaged weighted empirical measures of the multi-particle systems,  all of which are  normal distributions for $N=1, 50, 100, 200$. Furthermore,  Figure \ref{figure3} plots their empirical density functions and the normal distribution $N(0,4/9)$, respectively.  One observes that as the particle number  $N$ increases, the shape of the empirical density function  becomes increasingly close to the density function of normal distribution $N(0,4/9)$.
\begin{figure}[H]
  \centering
\includegraphics[width=12cm,height=8cm]{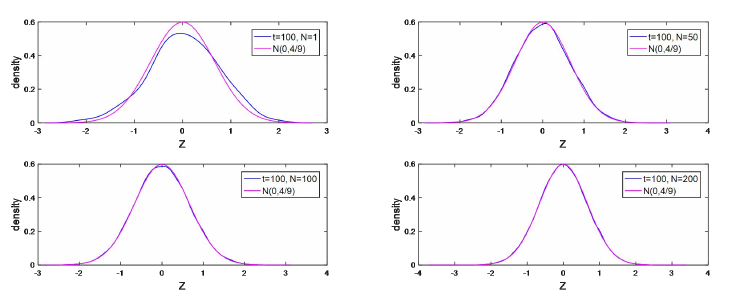}
\captionsetup{font=footnotesize}
  \caption{The empirical density functions of multi-particle system at $t=100$ for $N=1, 50, 100, 200$~and the normal distribution $N(0,4/9)$.}
\label{figure3}
\end{figure}
\end{expl}


\end{document}